\documentclass[12pt]{article}
\usepackage[margin=1in]{geometry}

\usepackage{amsmath, amsfonts,amssymb, mathtools}
\usepackage{amsthm}
\usepackage{enumerate, bbm}
\usepackage{graphicx}
\usepackage{authblk}
\usepackage{multirow}
\usepackage{booktabs}
\usepackage{xcolor}

\usepackage{natbib}
\usepackage[colorlinks,citecolor=blue,urlcolor=blue]{hyperref}

\usepackage{caption}
\usepackage{subcaption}

\usepackage{setspace}
\usepackage{parskip}
\newtheoremstyle{paperthm}
  {1em}   
  {1em}   
  {\itshape} 
  {}      
  {\bfseries} 
  {.}     
  {.5em}  
  {}      

\theoremstyle{paperthm}

\newtheorem{theorem}{Theorem}
\newtheorem{lemma}{Lemma}

\newtheorem{corollary}{Corollary}
\newtheorem{proposition}{Proposition}

\newcommand{\ind}{\mathbbm{1}}
\newcommand{\E}{\mathbb{E}}
\newcommand{\cE}{\mathcal{E}}

\newcommand{\cR}{\mathcal{R}}
\newcommand{\cH}{\mathcal{H}}

\newcommand{\var}{\textnormal{Var}}
\newcommand{\bigO}{\mathcal{O}}

\newcommand{\norm}[1]{\left\lVert#1\right\rVert}
\newcommand{\Bin}{\textnormal{Bin}}

\DeclarePairedDelimiterX{\infdivx}[2]{(}{)}{%
  #1\;\delimsize\|\;#2%
}

\DeclarePairedDelimiter\floor{\lfloor}{\rfloor}

\begin{document}

\title{Learning Smooth Populations of Parameters \\ with Trial Heterogeneity}
\author[1]{JungHo Lee\thanks{\url{junghol@andrew.cmu.edu}}}
\author[2]{Valerio Ba\'{c}ak\thanks{\url{valerio.bacak@rutgers.edu}}}
\author[1]{Edward H. Kennedy\thanks{\url{edward@stat.cmu.edu}}}
\affil[1]{Department of Statistics and Data Science, Carnegie Mellon University}
\affil[2]{School of Criminal Justice, Rutgers University}
\date{\today}


\maketitle

\begin{abstract}
We consider the classical problem of estimating the mixing distribution of binomial mixtures, but under \emph{trial heterogeneity} and \emph{smoothness}.
This problem has been studied extensively when the trial parameter is homogeneous, but not under the more general scenario of heterogeneous trials, and only within a low smoothness regime, where the resulting rates are slow. 
Under the assumption that the density is $s$-smooth, we derive fast error rates for kernel density estimators under trial heterogeneity that depend on the harmonic mean of the trials.
Importantly, even when reduced to the homogeneous case, our result improves on the state-of-the-art rate of \cite{ye2021binomial}.
We also study nonparametric estimation of the difference between two densities, which can be smoother than the individual densities, in both i.i.d. and binomial-mixture settings.
Our work is motivated by an application in criminal justice: comparing conviction rates of indigent representation in Pennsylvania. 
We find that the estimated conviction rates for appointed counsel (court-appointed private attorneys) are generally higher than those for public defenders, potentially due to a confounding factor: appointed counsel are more likely to take on more severe cases. 
\end{abstract}

\vfill

\newpage
\section{Introduction}
The problem of estimating the mixing distribution under binomial mixtures has received longstanding attention in the statistics literature \citep{teicher1963identifiability, lord1967estimating, lord1975empirical, sivaganesan1993robust, wood1999binomial, tian2017learning, vinayak2019optimal, vinayak2019maximum, ye2021binomial}. 
In this problem, we observe $n$ binomial random variables, each with potentially unique probabilities and number of trials. Early works primarily addressed identifiability \citep{teicher1963identifiability, wood1999binomial}; 
recent studies focus on optimal estimation under different regimes \citep{tian2017learning, vinayak2019maximum, ye2021binomial} as well as testing \citep{kania2025testingrandomeffectsbinomial}. 

Surprisingly, most prior work has focused on the special case in which the number of trials is homogeneous, even though heterogeneous trials are ubiquitous. For instance, in our application to estimating the conviction rates by attorney type in Pennsylvania, the number of cases varies significantly across counties. \cite{kline2021reasonable}, which develops a model for detecting employer discrimination, is a notable exception. 

Even in contexts where the trial parameters are naturally heterogeneous, existing approaches largely remain ad hoc, for example by reducing the data to a homogeneous-trial setting through thresholding, as in \cite{tian2017learning}, leaving much of the additional information from the trials untapped. Trial heterogeneity arises more frequently in the Bayesian literature where trials are random and treated as parameters of interest themselves \citep{royle2004n, wenger2008estimating, dail2011models, wu2015bayesian}. By contrast, the binomial mixing distribution literature regards the trial parameter as fixed. As will be seen, we likewise condition on the realized $t_i$'s, imposing a strictly weaker assumption and gaining robustness to any mis-specification of the trial distribution.


The problem can also be viewed through the lens of density deconvolution and additive measurement error models, where one observes a noisy surrogate of a latent variable and seeks to recover its density. Foundational work developed deconvolution kernel estimators for additive measurement error models \citep{stefanski1990deconvolving}, while classical rate theory established that deconvolution is an inverse problem whose difficulty depends on the measurement error distribution \citep{carroll1988optimal, fan1991global, fan1992deconvolution}. Related work on nonparametric errors-in-variables regression develops analogous deconvolution ideas for regression function estimation with contaminated covariates \citep{fan1993nonparametric, carroll2006measurement}. Our setting differs from this classical formulation because the ``noise'' induced by binomial sampling is discrete, bounded, depends on the latent success probability, and importantly, is heterogeneous through the trial counts. Related heteroscedastic deconvolution problems have been studied in continuous additive-error settings \citep{delaigle2008density}, but the binomial structure considered here motivates a separate analysis.

In particular, while prior work on binomial mixtures has largely relied on minimal structural assumptions, such as low smoothness \citep{ye2021binomial}, we show that higher-order smoothness can be leveraged to obtain faster rates. Such additional assumptions could yield significantly faster rates and thus be valuable in addressing these measurement error problems where known rates are generally slow \citep{carroll1988optimal, fan1991global, fan1992deconvolution, fan1993nonparametric}.

In this paper, we thus study nonparametric estimation of the binomial mixing density under \emph{trial heterogeneity} and \emph{smoothness}. 
We derive improved error rates using a kernel density estimator under $s$-H\"{o}lder smoothness, a standard structural assumption in the nonparametric literature \citep{tsybakov2008introduction}.
Our analysis covers both the usual KDE based on empirical proportions and a centered variant; in the homogeneous case, these bounds sharpen the rates available from the Lipschitz-level analysis of \cite{ye2021binomial}, with the centered estimator yielding an additional gain in the higher-smoothness regime.

Further, we analyze the  \emph{density difference}. 
This quantity is potentially of broader interest and has found widespread applications in areas including, but not limited to, change point detection \citep{kawahara2012sequential} and feature selection \citep{torkkola2003feature}. We study nonparametric estimation of the density difference when it is $\gamma$-smooth, where $\gamma \geq s$. We demonstrate that one can leverage standard nonparametric tools to optimally estimate this quantity with proper tuning.

We illustrate our method in an application in criminal justice. 
Specifically,  we compare the conviction rates of different forms of indigent representation -- appointed counsel (court-appointed private attorneys) and public defenders -- by estimating their densities. 
Our data include the number of convictions (``successes") and total cases (``trials") for each county in Pennsylvania, which differ significantly across counties and attorney types and thus motivates the study of our problem.
We find that the estimated conviction rates for appointed counsel are generally higher than those for public defenders, potentially due to a confounding factor: appointed counsel are more likely to take on more severe cases. 
This suggests a need for further investigation into how appointed counsel are systemically assigned to cases in Pennsylvania.

\subsection{Setup \& Notation}
\label{sec:setup}


We assume independent data $X_1,\ldots,X_n$ where $X_i \mid Q_i=q \sim \Bin(t_i,q)$, $Q_i$ is drawn from the mixing density $p$ supported on $[0,1]$, and $t_i$ is fixed and known, for $i=1,\ldots,n$. Under the classical setting with homogeneous trials, we would replace $t_i$ with $t$. 

Our goal is to study the estimation of $p$ at a point $u \in (0,1)$ with absolute pointwise bias $|\E\{\widehat{p}(u)\}-p(u)|$ and variance $\var\{\widehat{p}(u)\}$. We first consider the canonical \emph{kernel density estimator} (KDE) 
\begin{equation}\label{eq:kde}
    \widehat{p}_h(u) \coloneqq \frac{1}{n}\sum_{i=1}^n\frac{1}{h}K\bigg(\frac{X_i/t_i-u}{h}\bigg),
\end{equation}
constructed from empirical proportions $X_i/t_i$, for a valid kernel $K$ and some bandwidth $h>0$. Thus, $\widehat p_h(u)$ treats the empirical proportion $X_i/t_i$ as a noisy surrogate for the latent success probability $Q_i$. Later, we also consider a centered version based on $(X_i+1)/(t_i+2)$. At times, we write 
$K_h(X_i/t_i) \coloneqq K\{(X_i/t_i - u)/h\}/h$ following the notational convention.

Our key assumption is that $p$ is $s$-smooth, i.e., $p$ belongs to a H\"{o}lder class with index $s$. More precisely, let $\floor{s}$ denote the largest integer strictly smaller than $s$. Then the H\"{o}lder class indexed by $s$ contains all functions $f$ that are $\floor{s}$ times continuously differentiable with derivatives up to order $\floor{s}$ bounded, and whose $\floor{s}$-th order derivatives are H\"older continuous of order $s-\floor{s}$. In particular, $s-\floor{s}=1$ when $s$ is an integer, so the last condition reduces to Lipschitz continuity of the $(s-1)$-st derivative. That is,
\begin{equation*}
    |\partial^j f(u)| \leq L \text{ for all } j \leq \floor{s}, \quad \text{and} \quad |\partial^{\floor{s}} f(u) - \partial^{\floor{s}} f(u')| \leq L|u-u'|^{s-\floor{s}},
\end{equation*} 
for some constant $L>0$ and for all $u, u' \in [0,1]$. 

We use $\|p\|_\infty \coloneqq \sup_{u} |p(u)|$ for the sup norm of a real-valued function $p$. For two real sequences $a_n$ and $b_n$ and for some $C>0$ independent of $n$, we write $a_n \lesssim b_n$ if $a_n \leq Cb_n$, $a_n \gtrsim b_n$ if $a_n \geq Cb_n$, and $a_n \asymp b_n$ if both $a_n \lesssim b_n$ and $a_n \gtrsim b_n$. We further write $\overset{d}{\to}$ and $\overset{p}{\to}$ to denote convergence in distribution and probability, respectively.

\paragraph{Paper outline.} The remainder of the paper is organized as follows. In Section \ref{sec:theory}, we provide our main theoretical results, including the error bounds for estimating $s$-smooth mixing densities under trial heterogeneity. In Section \ref{sec:den_diff}, we study the density difference. In Sections \ref{sec:simulation} and \ref{sec:application}, we perform relevant simulation studies and discuss our application to estimating conviction rates by attorney type, respectively. Finally, we conclude our work with potential directions for future research in Section \ref{sec:conclusion}.

\section{Main Theoretical Results}\label{sec:theory}
We begin by presenting a key lemma that underpins the majority of our analysis. 

\subsection{Error Bounds for Smooth Mixing Densities}\label{sec:error_bounds}
Central to our analysis is bounding the bias that arises from using empirical proportions $X_i/t_i$ instead of true proportions $Q_i$. To do so, one needs to consider an error of the form
\begin{equation}\label{eq:p_weighted_bernstein_error}
    \cE(f_i) \coloneqq 
    \int \bigg\{\sum_{x=0}^{t_i}f(x/t_i) \binom{t_i}{x} q^{x}(1-q)^{t_i-x}-f(q)\bigg\}p(q)dq,
\end{equation}
i.e., $\int\big[\E\{f(X_i/t_i) \mid Q_i=q\}-f(q)\big]p(q)dq$ for some function $f$. More precisely, \eqref{eq:p_weighted_bernstein_error} is the $p$-weighted integrated error of a Bernstein polynomial approximation of $f$ of degree $t_i$.\footnote{The subscript $i$ on $f$ in the right side of \eqref{eq:p_weighted_bernstein_error} is meant to highlight the dependence on $i$.}

Although we are primarily interested in the case where $f=K_h$, we first prove a powerful lemma that controls the average of the Bernstein error \eqref{eq:p_weighted_bernstein_error} for a \emph{generic} $f$.

\begin{lemma}[Bound on Bernstein approximation error]\label{lem:bernstein_error_bound}
Assume that
\begin{enumerate}
    \item $p$ is $s$-smooth,
    \item $\norm{p}_\infty \leq p_{max}$ for some  $0 < p_{\max} < \infty$, and  \label{cond:p_bdd}
    \item $|p(u) - p(u')| \leq L|u-u'|^\alpha$ for some $0 < L < \infty$, $0 < \alpha \leq 1$ and all $u, u' \in [0,1]$. \label{cond:p_holder_cont}
\end{enumerate} Then
\begin{equation}\label{eq:bound_bern_lemma}
     \left|\frac{1}{n}\sum_{i=1}^n\cE(f_i)\right| \leq \frac{1}{n}\sum_{i=1}^n \left\{g(t_i,\alpha)\sum_{x=0}^{t_i}\frac{|f(x/t_i)|}{t_i+1} + \cR(f_ip)\right\},
 \end{equation}
where
\begin{equation*}\label{eq:g}
    g(t_i,\alpha) \coloneqq L\left(\frac{1/4}{t_i+3}\right)^{\alpha/2} + L\left(\frac{1}{t_i+2}\right)^\alpha + \frac{p_{\max}}{t_i},
\end{equation*}
and
 \begin{equation}\label{eq:quasi-riemann}
     \cR(f_ip) \coloneqq \left|\sum_{x=0}^{t_i}\frac{f(x/t_i)p(x/t_i)}{t_i} - \int f(q)p(q)dq \right|,
 \end{equation}
 is a quasi-Riemann sum approximation error to $\int f(q)p(q)dq$.\footnote{We use the term ``quasi-Riemann" since the sum is a left/right sum with one extra term.}
\end{lemma}
The key idea behind the proof of the above is to observe that
\begin{equation*}
    \int \sum_{x=0}^{t_i}f(x/t_i) \binom{t_i}{x} q^{x}(1-q)^{t_i-x}p(q)dq = \sum_{x=0}^{t_i} f(x/t_i)\frac{\E\{p(B_{xt_i})\}}{t_i+1},
\end{equation*} 
where $B_{xt_i}$ is a Beta random variable with parameters $x+1$ and $t_i-x+1$, and then decompose the Bernstein approximation error \eqref{eq:p_weighted_bernstein_error} as
\begin{align*}
     \cE(f_i) 
     &= \sum_{x=0}^{t_i} f(x/t_i)\frac{\E\{p(B_{xt_i})\} - p\{\E(B_{xt_i})\}}{t_i+1} + \sum_{x=0}^{t_i} f(x/t_i)\left\{\frac{p\{\E(B_{xt_i})\}}{t_i+1}-\frac{p(x/t_i)}{t_i}\right\} \\
     &+ \left\{\sum_{x=0}^{t_i} f(x/t_i)\frac{p(x/t_i)}{t_i} - \int f(q)p(q)dq \right\}, 
\end{align*}
so that the last term corresponds to the quasi-Riemann sum approximation error.

Notably, the error bound \eqref{eq:bound_bern_lemma} holds for any $f$ and simply relies on its discrete $L_1$ norm and the quasi-Riemann sum approximation error \eqref{eq:quasi-riemann}. On the other hand, classical Bernstein bounds \citep{Kac1938UneRS, lorentz1986approximation, lorentz2012bernstein, bojanic1989rate, mathe1000approximation} require smoothness of $f$ and its derivatives. With $f=K_h$, such bounds result in undesirable $1/h$ dependencies that lead to loose error bounds. Importantly, even under trial homogeneity with $f=K_h$, one can show that our Bernstein bound \eqref{eq:bound_bern_lemma} reduces to $\bigO(1/\sqrt{t} + 1/ht)$, which improves upon the bound $\bigO(h + 1/\sqrt{t} + 1/ht)$ in Theorem 8 of \cite{ye2021binomial} obtained using a histogram estimator.

Next, we present our first main theorem that characterizes upper bounds for pointwise bias and variance of the KDE \eqref{eq:kde} under trial heterogeneity and $s$-smoothness. The bias bound, in particular, consists of (i) the usual smoothing bias of order $h^s$ and (ii) additional terms stemming from the average Bernstein error bound \eqref{eq:bound_bern_lemma} by setting $f=K_h$. Interestingly, the bounds depend on the \emph{harmonic mean} of the trial parameters.

\begin{theorem}\label{thm:bound_on_error_kernel_het}
     Let $\tilde{t} \coloneqq n/\sum_{i=1}^n t_i^{-1}$ denote the harmonic mean of the trial parameters $t_1,\ldots,t_n$. Assume the conditions of Lemma \ref{lem:bernstein_error_bound}
     and additionally
     \begin{enumerate}
            \item $|K(v)| \leq K_{max}\ind(|v| \leq 1)$ for some $0 < K_{\max} < \infty$, \label{cond:K_bdd}
            \item $|K(v)-K(v')| \leq M|v-v'|^{\beta}$ for some $0 < M < \infty$, $0 < \beta \leq 1$, and all $v,v' \in [-1,1]$, \label{cond:K_holder_cont}
            \item $\int K(u)du = 1$, \label{cond:K_int_1}
            \item $\int u^jK(u)du = 0$ for $j = 1,\ldots,\floor{s}$, and \label{cond:K_higher_order}
            \item $\int |u|^s|K(u)|du \leq B$ for some $0 < B < \infty$. \label{cond:K_higher_bdd}
    \end{enumerate}
    Then, for $u \in (0,1)$, the pointwise bias of the kernel density estimator constructed from empirical proportions \eqref{eq:kde} is bounded as
    \begin{align*}
         |\E\{\widehat{p}_h(u)\}-p(u)| \leq \frac{LB}{\floor{s}!}h^s + \frac{1}{n}\sum_{i=1}^n \left\{g(t_i,\alpha) \sum_{x=0}^{t_i}\frac{|K_h(x/t_i)|}{t_i+1} + r(h, t_i, \alpha, \beta)\right\},
    \end{align*}
    where $r(h,t_i,\alpha, \beta)$ is an upper bound on the quasi-Riemann sum approximation error \eqref{eq:quasi-riemann} defined as
    \begin{equation*}\label{eq:r}
        r(h,t_i,\alpha, \beta) \coloneqq \frac{K_{\max}p_{\max}}{ht_i} + \left(2+\frac{1}{ht_i}\right)\left\{LK_{\max}\left(\frac{1}{t_i}\right)^\alpha + 2Mp_{\max}\left(\frac{1}{ht_i}\right)^\beta \right\}.
    \end{equation*}
    Furthermore, the pointwise variance is bounded as
    \begin{equation*}
        \var\{\widehat{p}_h(u)\} \leq \frac{K^2_{\max}p_{\max}}{nh}\bigg(2+\frac{1}{h\tilde{t}}\bigg).
    \end{equation*}
\end{theorem}
The assumptions we make in Theorem \ref{thm:bound_on_error_kernel_het} are standard in the kernel density estimation literature. The first states that the kernel is bounded and supported on $[-1,1]$, and the second that the kernel is H\"{o}lder continuous with exponent $\beta$ on the support. The rest of the assumptions define what is called a \emph{higher-order kernel} to achieve the smoothing bias of order $h^s$. See \cite{tsybakov2008introduction} for how to construct bounded kernels of arbitrary order using Legendre polynomials. 

The following corollary gives a simplified result to ease interpretability.

\begin{corollary}\label{cor:bound_on_error_kernel_het}
    Assume the conditions of Theorem \ref{thm:bound_on_error_kernel_het} and additionally
    \begin{enumerate}
        \item $p$ and $K$ are Lipschitz, \label{cond:p_K_lipschitz} and
        \item $C/\tilde{t} \leq h \leq C'$ for some $C,C' > 0$. \label{cond:bounded_h}
    \end{enumerate}
    Then the bounds in Theorem \ref{thm:bound_on_error_kernel_het} simplify to

     \begin{equation*}
        |\E\{\widehat{p}_h(u)\}-p(u)| \lesssim h^s + \frac{1}{\sqrt{\tilde{t}}} + \frac{1}{h\tilde{t}} \quad \text{and} \quad \var\{\widehat{p}_h(u)\} \lesssim \frac{1}{nh}.
    \end{equation*}
\end{corollary}
The second condition constrains $h$ to be in the ``right window" and caps $1/h\tilde{t}$ in both the bias and variance terms by a fixed constant.

Corollary \ref{cor:bound_on_error_kernel_het} encapsulates our key result: Under trial heterogeneity and $s$-smoothness, the pointwise bias and variance of the KDE \eqref{eq:kde} are of order $h^s + 1/\sqrt{\tilde{t}} + 1/h\tilde{t}$ and $1/nh$, respectively. When reduced to homogeneity, the bias bound is retained by replacing the harmonic mean $\tilde{t}$ with some homogeneous trial parameter $t$. As such, this result directly improves the state-of-the-art bias of order $h + 1/\sqrt{t} + 1/ht$ for a histogram estimator when the mixing distribution is assumed to be $1$-smooth \cite[Theorem 8]{ye2021binomial}. 

This refines the previous discussion in \citet[pg. 9]{ye2021binomial}. In their histogram-based analysis, the relevant error bound is $\mathcal O(h + 1/\sqrt{t} + 1/ht)$, so the $\mathcal O(h)$ term would dominate any higher-order smoothing bias $\mathcal O(h^s)$. Our improvement comes from combining the KDE with the Bernstein approximation error bound \eqref{eq:bound_bern_lemma} in Lemma \ref{lem:bernstein_error_bound}, which yields $\mathcal O(1/\sqrt{t} + 1/ht)$ instead. See Appendix A.3 for details.

\paragraph{Effect of severe trial heterogeneity.}
The harmonic mean $\tilde t$ in Corollary~\ref{cor:bound_on_error_kernel_het} gives a compact way to express the dependence arising from the explicit bounds in Theorem~\ref{thm:bound_on_error_kernel_het}, which depend on averages of inverse powers of the individual trial sizes. In particular, the empirical-proportion bias contains terms of the form
\[
    \frac1n\sum_{i=1}^n \frac{1}{\sqrt{t_i}},
    \quad \text{and} \quad
    \frac1{nh}\sum_{i=1}^n \frac{1}{t_i}.
\]
This implies that isolated observations with small $t_i$'s have limited impact because they contribute only order $1/n$ to the inverse-trial-size averages. However, if a non-negligible fraction of observations has small trial counts, these observations can dominate the additional binomial sampling bias. 
Thus, if trial heterogeneity is severe only in the sense that a small number of observations have very small trial counts, the KDE \eqref{eq:kde} can still be robust. On the other hand, if a non-negligible fraction of observations carry large binomial noise, its performance can deteriorate.

\subsection{Faster rates under higher-order smoothness via centering}
The KDE in \eqref{eq:kde} is constructed from the empirical proportions $X_i/t_i$. We now show that, in a high smoothness regime where $s \geq 2$, a small centering correction can improve the bias. The key observation is that the beta random variable appearing in the Bernstein representation has mean $(x+1)/(t_i+2)$ rather than $x/t_i$. Centering the estimator at this value removes the first-order term in the Taylor expansion of $p$, thereby improving the leading Bernstein approximation error from order $1/\sqrt{t_i}$ to order $1/t_i$.

For $x \in \{0,\ldots,t_i\}$, define $\mu_{i,x} \coloneqq (x+1)/(t_i+2)$ and consider the \emph{centered} $p$-weighted integrated Bernstein approximation error
\[
\cE_\mu(f_i) \coloneqq \int\left\{\sum_{x=0}^{t_i}f(\mu_{i,x})\binom{t_i}{x}q^x(1-q)^{t_i-x} - f(q)\right\}p(q)dq .
\]
This differs from $\cE(f_i)$ in \eqref{eq:p_weighted_bernstein_error} only in that $f_i$ is evaluated at $\mu_{i,x}$ rather than at $x/t_i$. We now present an analogue of Lemma~\ref{lem:bernstein_error_bound} for $\mathcal E_\mu(f_i)$.

\begin{lemma}\label{lem:centered_bernstein_error_bound}
Assume that
\begin{enumerate}
\item $p$ is $s$-smooth with $s \geq 2$, and
\item $|p'(u)-p'(u')| \leq L' |u-u'|$
for some $0<L'<\infty$ and all $u,u'\in[0,1]$.
\end{enumerate}
Then
\begin{equation*}
\left|\frac1n \sum_{i=1}^n \cE_{\mu}(f_i)\right|
\leq \frac1n \sum_{i=1}^n \left\{\frac{L'}{8(t_i+3)}\sum_{x=0}^{t_i}\frac{|f_i(\mu_{i,x})|}{t_i+1} + \cR_\mu(f_ip)\right\},
\end{equation*}
where
\begin{equation*}
\cR_\mu(f_ip) \coloneqq
\left|\sum_{x=0}^{t_i}\frac{f_i(\mu_{i,x})p(\mu_{i,x})}{t_i+1} - \int f_i(q)p(q)dq\right|.
\end{equation*}
\end{lemma}

The key improvement over Lemma \ref{lem:bernstein_error_bound} is that the coefficient in front of the discrete $L_1$ term $\sum_{x=0}^{t_i}|f_i(\mu_{i,x})|/(t_i+1)$ is now of order $t_i^{-1}$ rather than $t_i^{-1/2}$. This comes from the identity
\[\E(B_{x,t_i})=\mu_{i,x}, \quad B_{x,t_i}\sim \mathrm{Beta}(x+1,t_i-x+1).\]
Thus the first-order term in the Taylor expansion of $p(B_{x,t_i})$ around $\mu_{i,x}$ vanishes, and the leading contribution is governed by
\[
\var(B_{x,t_i}) = \frac{(x+1)(t_i-x+1)}{(t_i+2)^2(t_i+3)} \le \frac{1}{4(t_i+3)}.\]

We now define the KDE constructed from the \emph{centered} empirical proportions, or ``centered KDE'' for short, as
\begin{equation}\label{eq:centered_kde}
        \widehat{p}_{\mathrm{c},h}(u) \coloneqq \frac{1}{n}\sum_{i=1}^n\frac{1}{h}K\bigg\{\frac{(X_i+1)/(t_i+2)-u}{h}\bigg\},
\end{equation}
The following corollary gives the corresponding simplified bias and variance bounds.

\begin{corollary}\label{cor:bound_on_error_kernel_het_centered}
Assume the conditions of Theorem \ref{thm:bound_on_error_kernel_het}, Lemma \ref{lem:centered_bernstein_error_bound}, and Corollary~\ref{cor:bound_on_error_kernel_het}. 
Then, for $u \in (0,1)$, the pointwise bias and the variance of the centered kernel density estimator \eqref{eq:centered_kde} are bounded as    
\begin{equation*}
        |\E\{\widehat{p}_{\mathrm{c},h}(u)\}-p(u)| \lesssim h^s + \frac{1}{\tilde{t}} + \frac{1}{h\tilde{t}} \quad \text{and} \quad \var\{\widehat{p}_{\mathrm{c},h}(u)\} \lesssim \frac{1}{nh}.
\end{equation*}
\end{corollary}

Corollary \ref{cor:bound_on_error_kernel_het_centered} shows that, under higher-order smoothness, one can have the additional bias terms to be of order $1/\tilde t + 1/(h\tilde t)$ rather than $1/\sqrt{\tilde t} + 1/(h\tilde t)$. It implies that the $1/\sqrt{\tilde t}$ term in Corollary \ref{cor:bound_on_error_kernel_het} is not intrinsic to the binomial mixture problem and instead arises from using the uncentered empirical proportions $X_i/t_i$. By evaluating the kernel at $(X_i+1)/(t_i+2)$, the estimator becomes ``aligned'' with the mean of the beta distribution appearing in the Bernstein error, thereby making the leading binomial approximation bias second order. 

The relevant distinction is the explicit smoothness threshold $s \geq 2$. Hence the original ``uncentered'' KDE remains useful under lower smoothness assumptions, while the centered KDE is preferable when $p$ is assumed at least twice smooth.

\subsection{Oracle Minimax Estimation}\label{sec:minimax}
We now discuss conditions on the harmonic mean $\tilde{t}$ required for \emph{oracle minimax estimation}, which we define as follows. In the canonical setting where we observe the i.i.d. true proportions, the minimax rate for estimating $p$, a 1-dimensional $s$-smooth density, is $n^{-s/(2s+1)}$ \citep{vandervaart1998asymptotic, tsybakov2008introduction}. We call this the \emph{oracle $s$-smooth minimax density estimation rate}, or just the \emph{oracle rate} for short.

We first consider the KDE \eqref{eq:kde}, which remains applicable in a low-smoothness regime of $s < 2$. To achieve the oracle rate, we need the bias induced by empirical proportions to be dominated by the variance, i.e.,
\[
    \frac{1}{\sqrt{nh}}
    \gtrsim
    \max\left\{
        \frac{1}{\sqrt{\tilde t}},
        \frac{1}{h\tilde t}
    \right\},
\]
so that we can balance the usual smoothing bias and variance. If $h \sim n^{-1/(2s+1)}$ by the usual balancing so that the oracle rate can be achieved, this condition translates into
\begin{equation}\label{cond:minimax_het}
    \tilde{t} \gtrsim \begin{cases}
        n^{\frac{1+1/s}{2+1/s}}, & s \leq 1, \\
        n^{\frac{2}{2+1/s}}, & s > 1.
    \end{cases}
\end{equation}   
When $s=1$, this matches $t \gtrsim n^{2/3}$ from \cite{ye2021binomial} in the homogeneous case. For $s>1$, the condition is driven by the requirement $1/\sqrt{nh} \gtrsim 1/\sqrt{\tilde t}$; hence, as $s\to\infty$, the uncentered KDE requires $\tilde t \gtrsim n$ to recover the oracle rate. Thus, for the uncentered KDE, exploiting higher-order smoothness improves the payoff through the smoothing bias $h^s$, but does not reduce the trial-size requirement. 

On the other hand, the improved bias bound from Corollary \ref{cor:bound_on_error_kernel_het_centered} relaxes this requirement for the centered KDE \eqref{eq:centered_kde}. To achieve the oracle rate, we need
\[
    \frac{1}{\sqrt{nh}}
    \gtrsim
    \max\left\{
        \frac{1}{\tilde t},
        \frac{1}{h\tilde t}
    \right\}.
\]
With $h\sim n^{-1/(2s+1)}$, the second term is dominant, and hence the condition becomes
\begin{equation}\label{cond:minimax_het_centered}
    \tilde t
    \gtrsim
    n^{\frac{1+1/s}{2+1/s}}.
\end{equation}

\paragraph{Remark on minimax rates.} 
To the best of our knowledge, the minimax rates for estimating binomial mixing densities under smoothness and trial heterogeneity are unknown. Existing minimax results for binomial mixtures primarily concern estimation or testing under Wasserstein-1 distance \citep{tian2017learning, vinayak2019maximum, kania2025testingrandomeffectsbinomial}, all under trial homogeneity.

A sharp minimax theory under smoothness would likely require estimators different from these KDEs based on empirical proportions. Classical deconvolution theory shows that optimal rates depend jointly on the smoothness
of the target density and the ill-posedness of the convolution operator induced by the
measurement error \citep{carroll1988optimal, fan1991global, fan1992deconvolution}; analogously, our observations reveal $p$ only through the binomial mixture operator $p\mapsto q_i$ where
\[
q_i(x)
=
\int_0^1 \binom{t_i}{x}q^x(1-q)^{t_i-x}p(q)dq,
\qquad x=0,\ldots,t_i .
\]
Thus, an optimal estimator would likely treat the problem as an inverse problem, estimating only the recoverable components of $p$ and choosing the ``resolution'' to balance approximation bias against the variance inflation induced by the binomial mixture operator. In the heterogeneous setting, this resolution should depend on the full trial-size profile. Roughly speaking, low-resolution features of $p$ should be informed by many observations, while high-resolution features are informed mainly by observations with large $t_i$'s.

A matching lower bound would require showing that sufficiently fine features of $p$ cannot be reliably recovered from the binomial observations. Prior minimax analyses for homogeneous binomial mixtures rely on moment-matching lower-bound constructions under Wasserstein-1 distance \citep{tian2017learning, vinayak2019maximum}, but extending these arguments to smooth pointwise density estimation under heterogeneous trials appears delicate. The least favorable perturbations must preserve H\"older smoothness while being nearly invisible in the induced binomial mixture probabilities across the full trial-size profile.

Heuristically, in the homogeneous case, the minimax rate over an $s$-H\"older class should interpolate between the ordinary rate $n^{-s/(2s+1)}$ when the trial size is sufficiently large and a rate of inverse-trial-size order  that reflects the binomial sampling barrier when the trial size is small. Between these extremes, we expect an intermediate regime in which the attainable resolution is governed jointly by $n$ and $t$, in line with Wasserstein results for binomial mixtures \citep{tian2017learning, vinayak2019maximum} and with classical deconvolution theory \citep{carroll1988optimal, fan1991global, fan1992deconvolution}. Under heterogeneous trials, the analogous rate should depend on the full trial-size profile. Developing this minimax theory is beyond the scope of the present paper and is an important direction for future work.

\subsection{Tuning Parameter Selection}\label{sec:tuning}
We now discuss tuning parameter selection for nonparametric density estimators constructed from empirical proportions using \emph{Lepski's method} \citep{lepski1991problem}. We consider a variant of the original approach proposed by \cite{spokoiny2019bootstrap}, which is succinctly described in \cite{chetverikov2024tuning}. The notation and discussion of the method closely follow \cite{chetverikov2024tuning}, but in the context of kernel density estimation, which is the main focus of this article. 

The core idea of the method is simple: Start with a large value of $h$, and gradually reduce it until further reduction of $h$ no longer significantly decreases the bias of $\widehat{p}_h(u)$, where the ``significance" is assessed in terms of variance. Formally, define
\begin{equation*}
    b_{h,h'} \coloneqq \E\{\widehat{p}_{h'}(u) - \widehat{p}_{h}(u)\},  \quad p_{h,h'} \coloneqq \var\{\widehat{p}_{h'}(u) - \widehat{p}_{h}(u)\},
\end{equation*}
for all $h, h' \in \cH$ and $\cH^-(h) = \{h' \in \cH: h' < h\}$. The goal is to test the null hypothesis
\begin{equation*}
    H_h: |b_{h,h'}|^2 \leq p_{h,h'} \quad \text{for all } h' \in \cH^-(h),
\end{equation*}
for each $h \in \cH$. Intuitively, if $H_h$ is not rejected, there is no evidence that a smaller bandwidth than $h$ would significantly reduce the bias. Now, a size $\alpha$ test of $H_h$ rejects if
\begin{equation*}
    \max_{h' \in \cH^-(h)} \frac{|\widehat{p}_{h'}(u) - \widehat{p}_{h}(u)|}{\sqrt{p_{h,h'}}} > 1 + c_h(\alpha),
\end{equation*}
where $c_h(\alpha)$ is the $(1-\alpha)$ quantile of 
\begin{equation*}
    \max_{h' \in \cH^-(h)} \frac{|\widehat{p}_{h'}(u) - \E\{\widehat{p}_{h'}(u)\} - \widehat{p}_{h}(u) + \E\{\widehat{p}_{h'}(u)\}|}{\sqrt{p_{h,h'}}}.
\end{equation*}
To implement the method in practice, one can replace $p_{h,h'}$ with
\begin{equation*}
    \widehat{p}_{h,h'} = \frac{1}{n(n-1)}\sum_{i=1}^n\bigg[\big\{K_{h'}(X_i/t_i) - K_{h}(X_i/t_i)\big\} - \big\{\widehat{p}_{h'}(u) - \widehat{p}_h(u)\big\}\bigg]^2,
\end{equation*}
and estimate $c_h(\alpha)$ using, e.g., multiplier bootstrap. For details, we point the reader to \cite{chetverikov2024tuning} and \cite{spokoiny2019bootstrap}. 

Let $\widehat{h}$ denote the tuning parameter selected by the procedure and
\[h^* =
\max\left\{h\in\mathcal H: |b_{h,h'}|^2\leq p_{h,h'} \text{ for all } h'\in\mathcal H^-(h)\right\},
\]
be the \emph{oracle} tuning parameter. Further, define $\cH_{\mathrm{IN}} \coloneqq \{h \in \cH: h \geq h^*\} \backslash \cH^0$ with $\cH^0 = \{h \in \cH^+(h^*): |b_{h,h^*}|^2 > p_{h,h^*}\{1 + c_h(\alpha) + c^-_{h^*}(\alpha)\}^2\}$. The set $\cH_{\mathrm{IN}}$ is referred to as the \emph{insensitivity region} (IR; \citealp{spokoiny2019bootstrap}) that consists of bandwidths larger than or equal to $h^*$ whose bias relative to $h^*$ is not large enough to be reliably detected. Here $\mathcal H^+(h)=\{h'\in\mathcal H:h'>h\}$, and $c^-_{h^*}(\alpha)$ denotes the corresponding critical value for comparing $h^*$ with smaller bandwidths.

It follows that, with probability at least $1-2\alpha$, the following inequality holds:
\begin{equation}\label{eq:lepski_bound}
    |\widehat{p}_{\widehat{h}}(u) - \widehat{p}_{h^*}(u)| \leq \max_{h \in \cH_{\mathrm{IN}}} \sqrt{p_{h,h^*}} \{1 + c_h(\alpha)\}.
\end{equation}
The quantity on the right side of \eqref{eq:lepski_bound} is the \emph{adaptation price} of the method, determined by (i) $c_h(\alpha)$ of order $\sqrt{\log|\cH|} \leq \sqrt{\log n}$, (ii) $p_{h,h^*}$, and (iii) $|\cH_{\mathrm{IN}}|$, the size of the insensitivity region. In our binomial mixtures setting, $p_{h,h^*}$ is of order $(1/h^* - 1/h)/n$ under the conditions of Corollary \ref{cor:bound_on_error_kernel_het}. It thus follows that, heuristically, 
\begin{equation*}
    \max_{h \in \cH_{\mathrm{IN}}} \sqrt{p_{h,h^*}} \{1 + c_h(\alpha)\} \lesssim \begin{cases}
        \sqrt{\log n/(nh^*)}, & \text{large IR},\\
        \sqrt{\log n/n}, & \text{small IR}.
    \end{cases}
\end{equation*}
That is, the overall adaptation price under empirical proportions can still be of \emph{identical} order to that under true proportions. The following proposition formalizes this.

\begin{proposition}\label{prop:lepski}
Assume that the conditions of Corollary \ref{cor:bound_on_error_kernel_het} hold uniformly over $h \in \mathcal H$ and additionally
\begin{enumerate}
    \item $K(-1) = K(1) = 0$, and
    \item $\log |\cH| = \mathcal O(\log n)$.
\end{enumerate}
Define 
$\Delta_{\mathrm{IN}} \coloneqq \max_{h\in\mathcal H_{\mathrm{IN}}}\left(1/h^* - 1/h\right)$
with $\Delta_{\mathrm{IN}} \coloneqq 0$ if $\mathcal H_{\mathrm{IN}}=\{h^*\}$. Then, with probability at least $1-2\alpha$, 
\[
|\widehat p_{\widehat h}(u)-\widehat p_{h^*}(u)|
\lesssim \begin{cases}
        \sqrt{\log n/(nh^*)}, & \Delta_{\mathrm{IN}}\asymp 1/h^*,\\
        \sqrt{\log n/n}, & \Delta_{\mathrm{IN}}= \mathcal O(1).
\end{cases}
\]
\end{proposition}

The same construction can be used for the centered KDE \eqref{eq:centered_kde} with the oracle bandwidth and admissible region interpreted under the corresponding centered bias bound.

The quantity $\Delta_{\mathrm{IN}}$ summarizes the size of $\mathcal H_{\mathrm{IN}}$ on the reciprocal bandwidth scale. Since $h \geq h^*$ for $h\in\mathcal H_{\mathrm{IN}}$, $0 \leq 1/h^*- 1/h \leq 1/h^*$. Thus $\Delta_{\mathrm{IN}}\asymp 1/h^*$ represents a large insensitivity region, with bandwidths much larger than $h^*$ still indistinguishable from $h^*$, whereas $\Delta_{\mathrm{IN}}=\mathcal O(1)$ represents a small region in which all insensitive bandwidths remain close to $h^*$.

In practice, one can take $\mathcal H$ to be a grid of the form $\mathcal H= \{h_{\max}a^{-m}:m=0,\ldots,M, a>1\}$, with $h_{\max}$ denoting the largest bandwidth in the grid. For example, one may take $h_{\max}$ fixed, choose a fixed grid ratio such as $a=1.1$, and set $h_{\min}=h_{\max}n^{-A}$ for some fixed $A>0$. More generally, it suffices for the theoretical argument that $\log(1/h_{\min})=\mathcal O(\log n)$ so that 
\[ 
    M=\frac{\log(h_{\max}/h_{\min})}{\log a}=\mathcal O(\log n),
\] 
which ensures $\log|\mathcal H|=\mathcal O(\log n)$.

Now, what remains is the comparison of the adaptation price with the oracle estimation error rates. Recall that, for the KDE \eqref{eq:kde} applicable in a low-smoothness regime of $s < 2$, the oracle estimation error is of order $h^{*s} + 1/\sqrt{nh^*} + 1/\sqrt{\tilde{t}} + 1/(h^*\tilde{t})$. If we let $1/\sqrt{nh^*} \gtrsim \max\{1/\sqrt{\tilde{t}}, 1/(h^*\tilde{t})\}$
so that the usual i.i.d. oracle error dominates, then we recover the typical story of Lepski's method: The adaptation price is larger than the oracle error under a large IR and smaller under a small IR. In particular, if $h^* \sim n^{-1/(2s+1)}$, then we have
\begin{equation*}
    \tilde{t} \gtrsim \begin{cases}
        n^{\frac{1+1/s}{2+1/s}}, & s \leq 1, \\
        n^{\frac{2}{2+1/s}}, & s > 1,
    \end{cases}
\end{equation*}
which is the same condition as in the oracle minimax density estimation \eqref{cond:minimax_het}. If $h^*$ decreases at a slower rate, however, then we would be faced with a higher requirement. 

The same calculation applies for the centered KDE \eqref{eq:centered_kde} when $s \geq 2$. In this case, the oracle estimation error is of order
$h^{*s} + 1/\sqrt{nh^*} + 1/\tilde{t} + 1/(h^*\tilde{t})$.
Letting $1/\sqrt{nh^*} \gtrsim \max\{1/\tilde{t}, 1/(h^*\tilde{t})\}$ and taking $h^* \sim n^{-1/(2s+1)}$ yields, as in \eqref{cond:minimax_het_centered},
\[
\tilde t
\gtrsim
n^{\frac{1+1/s}{2+1/s}}.
\]
Overall, in the best-case scenario where the oracle bandwidth decreases at the usual optimal rate, the condition on $\tilde t$ needed to recover the usual i.i.d. behavior of Lepski's method matches the corresponding oracle minimax density estimation condition.

\paragraph{Why not cross validation?} An alternative for tuning parameter selection is the omnipresent cross-validation with the pseudo-risk
\begin{equation*}
    \int \widehat{p}_h(u)^2du -  \frac{2}{n}\sum_{j \in D_2^{n}} \widehat{p}_h(X_j/t_j),
\end{equation*} 
and its minimizer $\widehat{h}$. We do not consider this approach for two reasons. First, 
analyzing the adaptation price of cross-validation under empirical proportions can be nontrivial due to the random nature of $\widehat{h}$. More importantly, the bar on $\tilde{t}$ for the usual i.i.d. cross-validation story (for details, see Chapter 8 of \cite{gyorfi2002distributionfree}) to hold may be much higher than it is for Lepski's method. 

In particular, based on our theoretical results in Section \ref{sec:error_bounds}, we roughly expect the following for the KDE \eqref{eq:kde}. Even in the ``best case" scenario where both the data-driven and the oracle tuning parameters decrease at the optimal rate of $n^{-1/(2s+1)}$, the condition on $\tilde{t}$ might be
\begin{equation*}
    \tilde{t} \gtrsim \begin{cases}
        n, & s \leq 1/2, \\
        n^{4/(2+1/s)}, & s > 1/2,
    \end{cases}
\end{equation*}
which is already a more stringent condition than the one for minimax oracle density estimation \eqref{cond:minimax_het}, or even worse,
\begin{equation*}
    \tilde{t} \gtrsim
    \begin{cases}
        n^{(1+1/s)/(2+1/s)}, & s \leq 1/3, \\
        n^{4/(2+1/s)}, & 1/3 < s \leq 1/2.
    \end{cases}
\end{equation*}
That is, recovering the usual story is only feasible in a low-smoothness regime. We thus leave cross-validation under empirical proportions for future work.

\subsection{Inference}\label{sec:inference}

One can adopt \emph{undersmoothing} \citep{hall1992bootstrap, chen2002confidence, chen2017tutorialkerneldensityestimation} to construct an asymptotically valid $1-\alpha$ confidence interval. We present the result for both the uncentered KDE \eqref{eq:kde} and the centered KDE \eqref{eq:centered_kde}. Let $Z_i^{\mathrm{uc}} = X_i/t_i$ and $Z_i^{\mathrm{c}} = (X_i+1)/(t_i+2)$. For this section only, we use the unified notation
\[
\widehat p_{\star,h}(u) = \frac1n\sum_{i=1}^n K_h(Z_i^\star),
\quad
\star\in \{\mathrm{uc},\mathrm{c}\}.
\]
The corresponding variance estimator is
\[
\widehat{\var}\{\widehat p_{\star,h}(u)\}
=
\frac{1}{n(n-1)}
\sum_{i=1}^n
\{K_h(Z_i^\star)-\widehat p_{\star,h}(u)\}^2.
\]

\begin{proposition}\label{prop:clt}
Assume the conditions of Theorem \ref{thm:bound_on_error_kernel_het}, and the conditions of Lemma \ref{lem:centered_bernstein_error_bound} for $\star=\mathrm{c}$. Additionally, assume that
\begin{enumerate}
\item $p$ and $K$ are Lipschitz,
\item  $nh\to\infty$, \label{cond:clt_h_nh} 
\item $nh\var\{\widehat p_{\star,h}(u)\} \to \sigma_\star^2$ for some $0 < \sigma_\star^2 < \infty$, and
\item 
\begin{enumerate}
    \item $\sqrt{nh}\{h^s+\tilde t^{-1/2}+(h\tilde t)^{-1}\}\to0$ for $\star = \mathrm{uc}$, 
    \item $\sqrt{nh}\{h^s+\tilde t^{-1}+(h\tilde t)^{-1}\}\to0$ for $\star = \mathrm{c}$.
\end{enumerate}
\end{enumerate}
It follows that
\[
\frac{\widehat p_{\star,h}(u)-p(u)}
{\sqrt{\var\{\widehat p_{\star,h}(u)}\}}
\overset{d}{\to}
N(0,1), \quad \text{and} \quad \frac{\widehat{\var}\{\widehat p_{\star,h}(u)\}}{\var\{\widehat p_{\star,h}(u)\}} \overset{p}{\to}1.
\]
Consequently,
$\widehat p_{\star,h}(u) \pm z_{1-\alpha/2} \sqrt{\widehat{\var}\{\widehat p_{\star,h}(u)}\}$
gives an asymptotically valid $1-\alpha$ confidence interval for $p(u)$.
\end{proposition}

The undersmoothing conditions in Proposition \ref{prop:clt} require the bias to be negligible relative to the standard error. For the uncentered KDE, they are equivalently implied by
\[
nh^{2s+1}\to0,
\qquad
\frac{\tilde t}{nh}\to\infty, \qquad \frac{\tilde t}{\sqrt{n/h}}\to\infty.
\]
If $h\sim n^{-\zeta}$ and $\tilde t\sim n^\kappa$ for some $\zeta, \kappa > 0$, these conditions become
\[
\max\left\{\frac{1}{2s+1}, 1-\kappa \right\} < \zeta < 2\kappa-1,
\]
with $\kappa>1/2$ (otherwise $\zeta < 0)$. Thus the bandwidth must be smaller than the oracle choice $h\asymp n^{-1/(2s+1)}$, but not too small with respect to how fast the harmonic mean $\tilde t$ grows. For the centered KDE, the analogous sufficient conditions are
\[
nh^{2s+1}\to0,
\qquad
\frac{\tilde t}{\sqrt{nh}}\to\infty, \qquad \frac{\tilde t}{\sqrt{n/h}}\to\infty.
\]
The last condition dominates the second when $h\to0$. Hence, if $h\sim n^{-\zeta}$ and $\tilde t\sim n^\kappa$, we have
\[
\frac{1}{2s+1} < \zeta < 2\kappa-1.
\]
The centered KDE removes the constraint $\zeta>1-\kappa$, allowing for a wider range of undersmoothing bandwidths. In particular, the interval is nonempty as soon as 
\[
    \kappa> (s+1)/(2s+1),
\] 
whereas for the uncentered KDE it is nonempty only if
\[
    \kappa>\max\left\{\frac{s+1}{2s+1},\frac{2}{3}\right\}.
\]
Therefore, within the higher-order smoothness regime where the centered KDE is applicable, centering relaxes the growth requirement on $\tilde t$, replacing the $2/3$ barrier by $(s+1)/(2s+1)$, which tends to $1/2$ as $s\to\infty$.

\section{Smooth Density Difference}\label{sec:den_diff}
Comparisons between two groups lie at the heart of statistics (e.g., two-sample testing) and science. In particular, we are often interested in the difference between two quantities, such as the difference in outcomes between the treated and control groups. To this end, we study nonparametric estimation of the \emph{density difference} for both true proportions $Q_i$ and empirical proportions $X_i/t_i$, assuming that each density is $s$-smooth. We first analyze the i.i.d. setting to show that the difference can be estimated optimally with standard non‑parametric techniques via proper tuning, and then turn to the binomial‑mixture setting, which is our main focus. The problem is analogous to the estimation of conditional average treatment effects in causal inference \citep{kennedy2023towards, kennedy2024minimax}, where one likewise seeks the difference in two regression functions.

\paragraph{True Proportions.} First, suppose that we have access to true proportions $Q_i$ and observe i.i.d. data $(Q_1, A_1),\ldots,(Q_n, A_n)$ where $A_i$ is a binary group indicator. The \emph{density difference} is defined as
\begin{equation}\label{eq:den_diff}
    \tau(u) \coloneqq p_1(u) - p_0(u),
\end{equation}
where $p_a(u) = p(u|A=a)$ for $a=0,1$ and $\int \tau(u)du = 0$. Importantly, $\tau(u)$ and $p_a(u)$ are likely to have different complexities; $\tau$ has to be at least as smooth as $p_a(u)$ and is potentially much more structured. 
Consider the following illustrative example adapted from \cite{kennedy2023towards} where the conditional densities are equal and given by
\begin{equation}\label{eq:nonsmooth_density}
p_a(u) = \frac{1}{C} 
\begin{cases}
    2u^2 + 2u, & 0 \leq u \leq 0.25, \\
    u, & 0.25 < u \leq 0.5, \\
    -u^2 + 2u + 0.2, & 0.5 < u \leq 0.75, \\
    1.5u - 1, &0.75 < u \leq 1,
\end{cases}
\end{equation}
with $C \approx 0.508$ denoting the normalizing constant. Clearly, each $p_a(u)$ is non-smooth, but the difference $\tau(u) = 0$ is $\infty$-smooth (Figure \ref{fig:diff_example}). 
\begin{figure}
    \centering
    \begin{subfigure}[b]{0.48\textwidth}
        \centering
        \includegraphics[width=\textwidth]{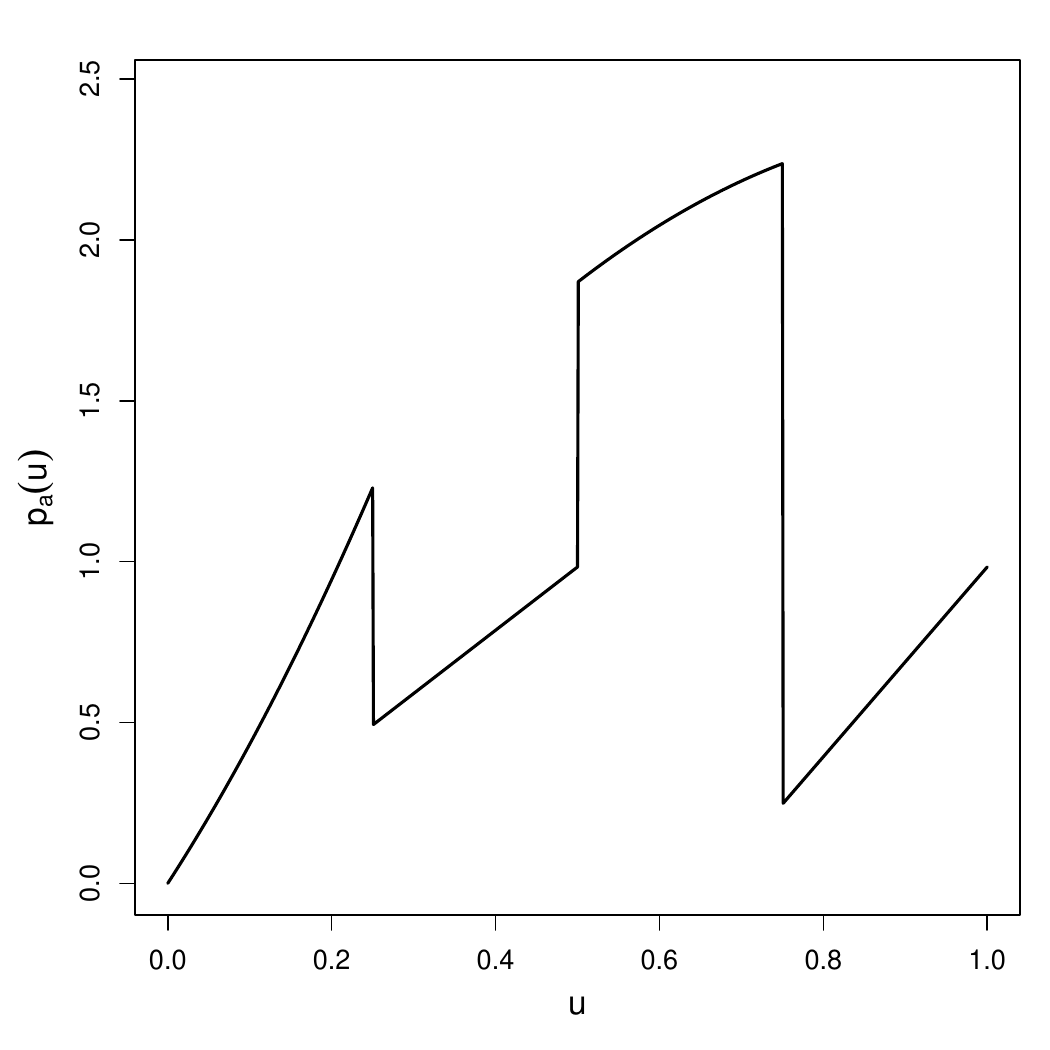}
    \end{subfigure}
    \hfill
    \begin{subfigure}[b]{0.48\textwidth}
        \centering
        \includegraphics[width=\textwidth]{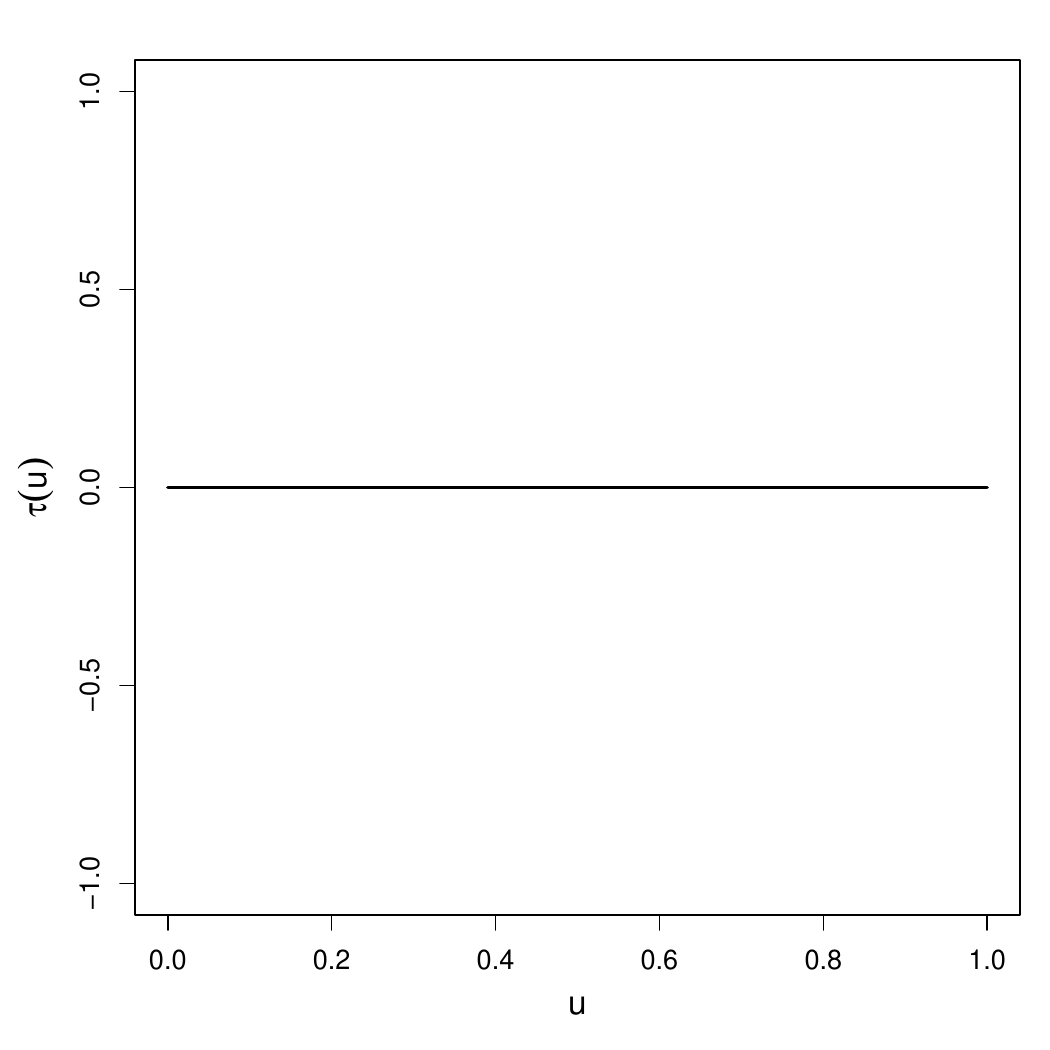}
    \end{subfigure}
    \caption{Plot of an illustrative example where the conditional densities individually are of lower smoothness (left), but their difference is of higher smoothness (right).}
    \label{fig:diff_example}
\end{figure}

An intuitive approach to estimate the density difference \eqref{eq:den_diff} is to estimate each density separately and take the difference, e.g.,
\begin{equation}\label{eq:den_diff_kernel}
    \tilde{\tau}_h(u) \coloneqq \tilde{p}_{1,h}(u) - \tilde{p}_{0,h}(u) \coloneqq \frac{1}{n}\sum_{i=1}^n\left(\frac{A_i}{\overline{A}}-\frac{1-A_i}{1-\overline{A}}\right)K_h(Q_i),
\end{equation}
where $\overline{A} \coloneqq \sum_{i=1}^n A_i/n$. The following proposition shows that the kernel density difference estimator \eqref{eq:den_diff_kernel} achieves the optimal smoothing bias of order $h^\gamma$ using an appropriate higher-order kernel.

\begin{proposition}\label{prop:den_diff}
 Assume that
 \begin{enumerate} 
     \item For $a=0,1$, \begin{enumerate}
         \item $p_a$ is $s$-smooth,
         \item $\|p_a\|_\infty \leq p_{\max}$ for some $0 < p_{\max} < \infty$,
     \end{enumerate}
         \item $\tau$ is $\gamma$-smooth with H\"older constant $0 < L_\tau < \infty$,
     \item $\epsilon \leq \sum_{i=1}^n A_i/n \leq 1-\epsilon$ for some $\epsilon > 0$ with probability 1,\label{cond:positivity}
    \item
    \begin{enumerate}
        \item $\int K(u)du = 1$, 
        \item $\int u^jK(u)du = 0$ for $j = 1,\ldots,\floor{\gamma}$,
        \item $\int |u|^{\gamma}|K(u)|du \leq B$ for some $0 < B < \infty$, and
        \item  $\int K(u)^2du \leq B'$ for some $0 < B' < \infty$. \label{cond:kernel_sq_int}
    \end{enumerate}
 \end{enumerate}
Then for $u \in (0,1)$, the pointwise bias and variance of the kernel density difference estimator \eqref{eq:den_diff_kernel} are bounded as
\begin{equation*}
    |\E\{\tilde\tau_h(u)\} - \tau(u)| \leq \frac{LB}{\floor{\gamma}!}h^\gamma \quad \text{and} \quad \var\{\tilde\tau_h(u)\} \leq \frac{1}{nh}\bigg(\frac{p_{\max}B'}{\epsilon^2}\bigg).
\end{equation*}
\end{proposition}
Similar results hold for other nonparametric methods. In the context of the orthogonal series estimator
\begin{equation*}
    b(u)^\top \int b(u)\tau(u)du \eqqcolon b(u)^\top\theta,
\end{equation*}
where $b(u) = \{b_1(u),\ldots,b_k(u)\}^\top$ denotes an orthonormal basis of dimension $k$, $\theta$ can be estimated unbiasedly with
\begin{equation*}
    \widehat\theta = \frac{1}{n}\sum_{i=1}^n\left(\frac{A_i}{\overline{A}}-\frac{1-A_i}{1-\overline{A}}\right)b(Q_i).
\end{equation*}
Therefore, we would again obtain the optimal bias order $k^{-\gamma}$ and the variance of order $k/n$ under standard assumptions. 

\paragraph{Tuning parameter selection.} 
In practice, performance depends critically on careful tuning. Tuning each density separately is sub-optimal, as will be shown in Section \ref{sec:sim2}. For simplicity, we consider cross-validation with two folds: $D_1^n$ to estimate the density difference and $D_2^n$ to choose the tuning parameter. 
For $h\in\cH$, define the $L_2$ risk 
\[ 
    R(h) \coloneqq \|\tilde\tau_h-\tau\|^2. 
\]
For optimal performance, one should minimize the ``joint" pseudo-risk \citep{wasserman2006all, gyorfi2002distributionfree}
\begin{equation}\label{eq:joint_pseudo_risk}
    \widehat R(\tilde{\tau}_h) \coloneqq 
    \int \tilde{\tau}_h(u)^2du -  \frac{2}{|D_2^n|}\sum_{j \in D_2^{n}} \left(\frac{A_j}{\overline{A}}-\frac{1-A_j}{1-\overline{A}}\right)\tilde{\tau}_h(Q_j),
\end{equation} 
where the average $\overline{A}$ is taken over $D_2^n$. The joint pseudo‑risk arises from decomposing the integrated squared error of the kernel density difference estimator
\begin{equation*}
    \int \tilde\tau_h(u)^2du -2 \int \tilde\tau_h(u)\tau(u)du + \int \tau(u)^2du,
\end{equation*}
and estimating the cross term $\int \tilde\tau_h(u)\tau(u)du$ by averaging $\tilde\tau_h(Q_j)$ over the second fold and weighting by $A_j/\overline{A} - (1-A_j)/(1-\overline{A})$, thereby turning the mixture average into the difference of group means. Hence, $\widehat R(\tilde\tau_h)$ estimates the integrated squared error up to the additive constant $\int \tau(u)^2du$ and the same finite-class hold-out oracle-inequality argument as in Theorem 7.1 of \cite{gyorfi2002distributionfree} can then be applied. 

\begin{proposition}\label{prop:joint_cv}
Assume that
\begin{enumerate}
    \item $\|p_a\|_\infty\leq p_{\max}$ for some $0 < p_{\max}<\infty$ for $a=0,1$,
    \item $\sup_{h\in\cH}\|\tilde\tau_h\|_\infty\leq 2p_{\max}$, and 
    \item $\epsilon\leq \overline A \leq 1-\epsilon$ for some $\epsilon>0$ with probability $1$.
\end{enumerate}
Let $h^* = \arg\min_{h\in\cH}R(h)$. Then, for every $\delta>0$,
\[
    \E\{R(\widehat h)\} \leq (1+\delta)\E\{R(h^*)\} + \bigg\{\frac{2p_{\max}(1+\delta)^2}{\delta\epsilon^2} + \frac{32}{3}(1+\delta)\bigg\}\left(\frac{p_{\max}}{\epsilon}+p_{\max}^2\right)\frac{1+\log|\cH|}{|D_2^n|}.
\]
\end{proposition}

Proposition \ref{prop:joint_cv} shows that joint cross-validation performs nearly as well as the best estimator in $\{\tilde\tau_h:h\in\mathcal H\}$ up to the remainder $\log|\mathcal H|/|D_2^n|$. Therefore, if $\mathcal H$ contains bandwidths and kernel orders that attain the optimal integrated squared error rate over a smoothness class, and if $\log |\cH| = \mathcal O(\log n)$, then the cross-validated estimator attains the same rate up to a negligible logarithmic term. In this sense, joint cross-validation allows adaptation to unknown smoothness provided the candidate set is sufficiently rich. Just as the Lepski's method in Section \ref{sec:tuning}, this can be achieved with a bandwidth grid combined with a set of kernel orders, e.g.,
\[
    \mathcal H = \{(h_m,r): h_m=h_{\max}a^{-m},\ m=0,\ldots,M, \ a>1, \ r\in\mathcal S\},
\]
where $\mathcal S$ is a collection of kernel orders. As before, if $h_{\min}=h_{\max}a^{-M}$ satisfies $\log (1/h_{\min}) = \mathcal O(\log n)$ and $\log|\mathcal S| = \mathcal O(\log n)$, then $\log|\mathcal H| = \mathcal O(\log n)$, so the cross-validation remainder remains of order $\log n/|D_2^n|$.

We emphasize that this idea of joint, or ``single-shot", tuning is well established in the literature \citep{hall1988nonparametric, sugiyama2013density, nguyen2014constrained}. Our contribution is to clarify that optimal estimation of the density difference in both theory and practice using standard nonparametric methods is \emph{feasible}, contrary to the conjecture that it may be difficult with $\gamma$ being unknown \citep[pg. 9]{sugiyama2013density}.\footnote{ \cite{sugiyama2013density} further discourages the use of higher-order kernels as they can be nonnegative, but negative density estimates can always be replaced with $0$ and the risk will never become worse.} 
See Section \ref{sec:sim2} for a simulation study that demonstrates the performance of the estimator \eqref{eq:den_diff_kernel} tuned with the joint pseudo-risk \eqref{eq:joint_pseudo_risk}.

\paragraph{Empirical proportions.} We now return to our main setting of interest where we observe independent data $(X_1, t_1, A_1),\ldots,(X_n, t_n, A_n)$.
Consider 
\begin{equation}\label{eq:den_diff_kernel_het}
    \widehat{\tau}_h(u) \coloneqq \widehat{p}_{1,h}(u) - \widehat{p}_{0,h}(u) \coloneqq \frac{1}{n}\sum_{i=1}^n\left(\frac{A_i}{\overline{A}}-\frac{1-A_i}{1-\overline{A}}\right)K_h(X_i/t_i),
\end{equation}
i.e., the empirical proportions version of the kernel density difference estimator \eqref{eq:den_diff_kernel}. The following corollary characterizes its error bounds.

\begin{corollary}\label{cor:den_diff_het}
Assume the conditions 1--3 of Proposition \ref{prop:den_diff} and the kernel conditions (1--5) in Theorem \ref{thm:bound_on_error_kernel_het}, with $s$ replaced by $\gamma$. Let \[
    \tilde t_a \coloneqq \frac{|\{i:A_i=a\}|}{\sum_{i:A_i=a} t_i^{-1}}, \quad a =0,1,
\]
be the group-specific harmonic mean and $A^n \coloneqq (A_1,\ldots,A_n)$. Then, for $u\in(0,1)$, the  pointwise bias of the kernel density difference estimator constructed from empirical proportions \eqref{eq:den_diff_kernel_het} conditional on $A^n$ is bounded as
\begin{equation*}
    |\E\{\widehat{\tau}_h(u) \mid A^n\}-\tau(u)|
    \leq
    \frac{L_a B}{\floor{\gamma}!}h^\gamma
    +
    \sum_{a=0,1}
    \frac{1}{|\{i:A_i=a\}|}\sum_{i:A_i=a}
    \bigg\{
        g_a(t_i,\alpha)
        \sum_{x=0}^{t_i}\frac{|K_h(x/t_i)|}{t_i+1}
        +
        r_a(h,t_i,\alpha,\beta)
    \bigg\},
\end{equation*}
where $g_a$ and $r_a$ are defined as in Theorem \ref{thm:bound_on_error_kernel_het} with $p$ replaced by $p_a$ and $L$ by by the corresponding H\"older constant $L_a$ for $p_a$. Further, the pointwise variance is bounded as
\begin{equation*}
    \var\{\widehat{\tau}_h(u) \mid A^n\}
    \leq
    \frac{K^2_{\max}p_{\max}}{nh\epsilon}
    \left(
        4+\frac{1}{h\tilde t_1}
        +\frac{1}{h\tilde t_0}
    \right).
\end{equation*}
\end{corollary}
Compared with Theorem \ref{thm:bound_on_error_kernel_het}, the smoothing bias is now governed by the smoothness $\gamma$ of the density difference $\tau$, while the additional bias from empirical proportions is governed by the two group-specific Bernstein approximation errors for $p_1$ and $p_0$. Thus, under mild simplifying conditions analogous to those in Corollary \ref{cor:bound_on_error_kernel_het}, we obtain
\begin{equation*}
|\E\{\widehat{\tau}_h(u)\mid A^n\}-\tau(u)|
\lesssim
h^{\gamma}
+ \sum_{a=0}^1
\left(\frac{1}{\sqrt{\tilde t_a}} + \frac{1}{h\tilde t_a}
\right),
\qquad
\var\{\widehat{\tau}_h(u)\mid A^n\}
\lesssim
\frac{1}{nh}.
\end{equation*}
Since $\epsilon \leq \overline A \leq 1-\epsilon$ with probability 1, the bias bound can also be written in terms of the global harmonic mean $\tilde t$ as
\begin{equation*}
|\E\{\widehat{\tau}_h(u) \mid A^n\}-\tau(u)|
\lesssim
h^{\gamma}
+
\frac{1}{\sqrt{\tilde t}}
+
\frac{1}{h\tilde t},
\qquad
\var\{\widehat{\tau}_h(u) \mid A^n\}
\lesssim
\frac{1}{nh},
\end{equation*}
which implies that the condition on $\tilde{t}$ to achieve the oracle rate of $n^{-1/(2+1/\gamma)}$ is
\begin{equation*}
    \tilde{t} \geq \begin{cases}
        n^{(1+1/\gamma)/(2+1/\gamma)}, & \gamma < 1, \\
        n^{2/(2+1/\gamma)}, & \gamma \geq 1.
    \end{cases}
\end{equation*}   

When $s \geq 2$, we can similarly define a centered empirical-proportion version of the density difference estimator,
\begin{equation}\label{eq:den_diff_kernel_het_centered}
\widehat{\tau}_{\mathrm{c},h}(u)
\coloneqq
\frac{1}{n}\sum_{i=1}^n
\left(\frac{A_i}{\overline{A}}-\frac{1-A_i}{1-\overline{A}}\right)
K_h\left(\frac{X_i+1}{t_i+2}\right).
\end{equation}
The following corollary gives a simplified error bound for this centered estimator.

\begin{corollary}\label{cor:den_diff_het_centered}
Assume the conditions of Corollary \ref{cor:den_diff_het}. Additionally, assume that the conditions of Corollary \ref{cor:bound_on_error_kernel_het_centered} hold for each $p_a$ with the bandwidth condition imposed groupwise, i.e., $C/\tilde t_a \leq h \leq C'$ for $a=0,1$. Then, for $u\in(0,1)$, the pointwise bias and variance of the centered kernel density difference estimator \eqref{eq:den_diff_kernel_het_centered} conditional on $A^n$ satisfy
\begin{equation*}
|\E\{\widehat{\tau}_{\mathrm{c},h}(u)\mid A^n\}-\tau(u)|
\lesssim
h^{\gamma}
+
\sum_{a=0,1}
\left(\frac{1}{\tilde t_a} + \frac{1}{h\tilde t_a}\right),
\quad
\var\{\widehat{\tau}_{\mathrm{c},h}(u)\mid A^n\}
\lesssim
\frac{1}{nh}.
\end{equation*}
\end{corollary}
The bound can again be written in terms of $\tilde t$ as
\begin{equation*}
|\E\{\widehat{\tau}_{\mathrm{c},h}(u)\mid A^n\}-\tau(u)|
\lesssim h^{\gamma}
+
\frac{1}{\tilde t}
+
\frac{1}{h\tilde t},
\qquad
\var\{\widehat{\tau}_{\mathrm{c},h}(u)\mid A^n\}
\lesssim
\frac{1}{nh},
\end{equation*}
which implies that the condition on $\tilde t$ to achieve the oracle rate $n^{-1/(2+1/\gamma)}$ becomes
\[
    \tilde t \gtrsim n^{(1+1/\gamma)/(2+1/\gamma)}.
\]
For tuning parameter selection under empirical proportions, we can again exploit Lepski's method, and analogous will guarantees hold.

\section{Simulation Studies}\label{sec:simulation}
In this section, we perform simulation studies to validate our theoretical results. 
\subsection{Leveraging Trial Heterogeneity}\label{sec:sim1}

Our first simulation quantifies the benefit of leveraging heterogeneous trial sizes by comparing the KDE based on empirical proportions \eqref{eq:kde} and the centered KDE \eqref{eq:centered_kde} with hypothetical benchmarks based on trial homogeneity.

The data generating process (DGP) is as follows. First, we draw the true proportions $Q_i \sim p$ where $p(u) = 6u(1-u)\ind(0 \leq u \leq 1)$, which is the Beta$(2,2)$ density. 
Next, for a target harmonic mean $\tilde{t}$, we sample $t_i \sim \textnormal{Poisson}(\tilde{t}$) then set $t_1 = t_{\min} \coloneqq 0.2\tilde{t}$. 
Then, we choose a random index $j \in \{2,\ldots,n\}$ and adjust $t_j$ up or down by an adaptive amount until the harmonic mean of $\{t_i\}_{i=1}^n$ is within $0.5$ of the target $\tilde{t}$. Finally, we generate the observed counts of ``successes" $X_i \sim \Bin(t_i, Q_i)$. 


We compare the KDE estimated using the empirical proportions \eqref{eq:kde} (``KDE'') and its centered analogue \eqref{eq:centered_kde} (``CKDE'') with homogeneous-trial benchmarks that replace each $t_i$ by a common trial size $t_0$. Specifically, for $t_0\in\{t_{\min},0.5\tilde t,\bar t\}$ with $\bar t \coloneqq \sum_{i=1}^n t_i/n$, we generate $Y_i \sim \Bin(t_0,Q_i)$ and compute
\[
    \widehat p_{t_0,h}(u) = \frac{1}{n}\sum_{i=1}^n K_h(Y_i/t_0).
\]
The estimator $\widehat{p}_{t_{\min},h}$ is therefore a thought experiment; it asks: ``\emph{What if we had ignored the extra Bernoulli outcomes already collected?}". The $t_{\min}$ benchmark is the most conservative since it retains all units while discarding all Bernoulli outcomes beyond the minimum trial count; the $\bar t$ benchmark is more favorable since it assigns the average trial count to every unit.

For a fair comparison, all estimators use the Epanechnikov kernel and the same bandwidth $h=n^{-1/5}$. We evaluate the pointwise bias and empirical standard error at $u=0.5$ of each estimator for target harmonic means $\tilde{t} \in \{30,35,40,\ldots,100\}$, averaged over 1,000 replications of the DGP.


\begin{figure}[ht!]
        \centering       
        \includegraphics[width=\textwidth]{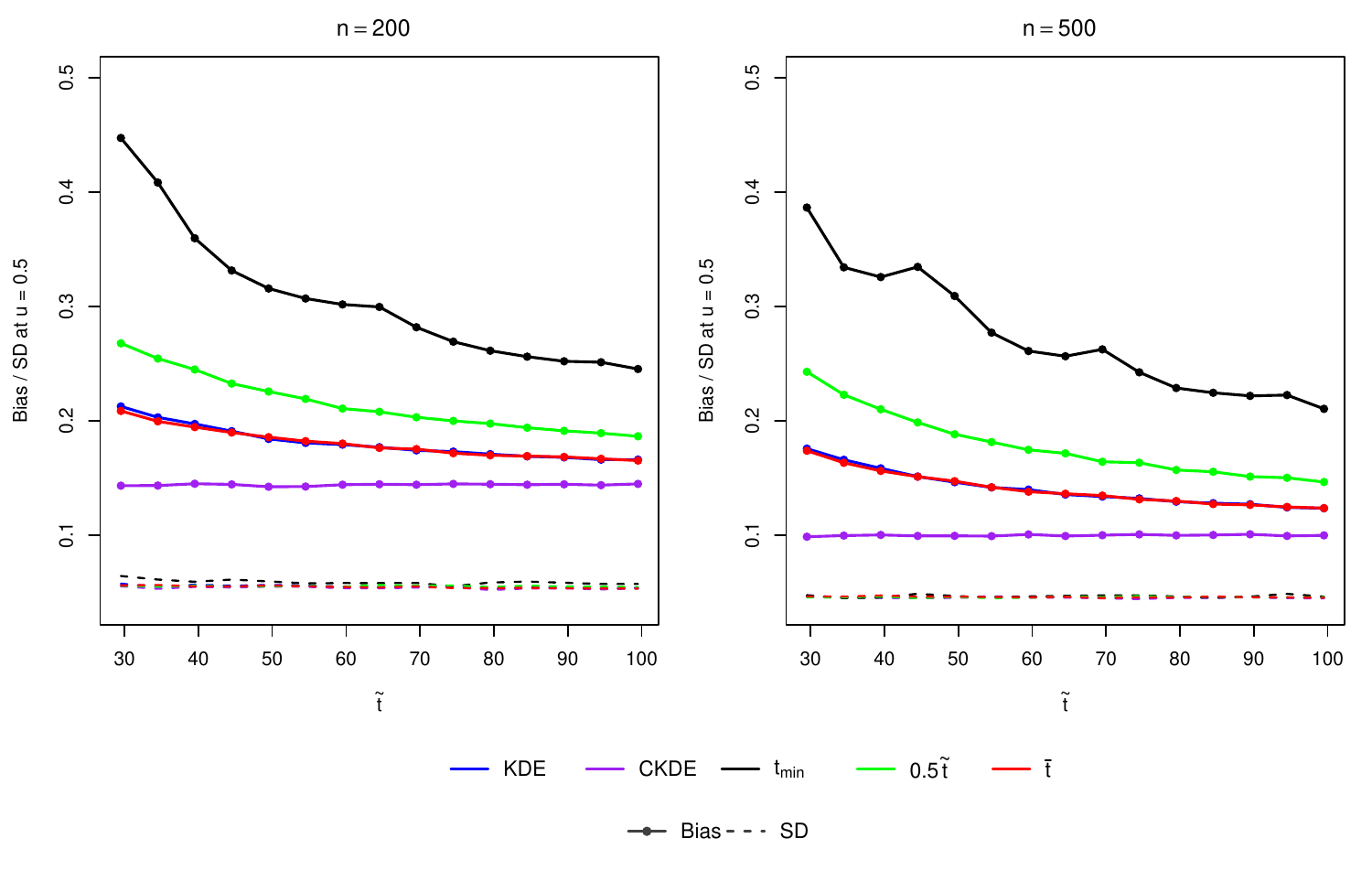}
    \caption{Pointwise bias and empirical standard error of the kernel density estimates.}
    \label{fig:sim1}
\end{figure}


Figure \ref{fig:sim1} displays the results. The heterogeneous KDE and CKDE achieve substantially lower pointwise bias than the homogeneous $t_{\min}$ and $0.5\tilde t$ benchmarks across all values of $\tilde t$, while having nearly identical empirical standard errors; as predicted by theory, CKDE achieves the lowest bias. The gap is largest when $\tilde t$ is small, illustrating that the gains from exploiting heterogeneous trial counts can be especially pronounced when trial sizes are limited. At the same time, the KDE performs similarly when $t_0 = \bar t$, because, in this simulation design, most trial counts are close to $\bar t$ and only a small fraction are close to $t_{\min}$. This is consistent with the discussion of severe trial heterogeneity in Section~\ref{sec:error_bounds}: a small $t_{\min}$ alone need not dominate performance, since the additional binomial sampling error depends on averages of inverse trial sizes rather than only on the minimum trial size.

Overall, the simulation supports the main message that, when trial sizes are heterogeneous, using the full trial-size profile can reduce binomial approximation bias without a meaningful variance penalty.

\subsection{Joint Tuning for Density Difference}\label{sec:sim2}
Next, we demonstrate the performance of the kernel density difference estimator \eqref{eq:den_diff_kernel} tuned with the joint pseudo-risk \eqref{eq:joint_pseudo_risk}. We revisit the nonsmooth density example \eqref{eq:nonsmooth_density} where 
\begin{equation*}
p_a(u) = \frac{1}{C} 
\begin{cases}
    2u^2 + 2u, & 0 \leq u \leq 0.25, \\
    u, & 0.25 < u \leq 0.5, \\
    -u^2 + 2u + 0.2, & 0.5 < u \leq 0.75, \\
    1.5u - 1, &0.75 < u \leq 1,
\end{cases}
\end{equation*}
for $a=0,1$ and therefore the density difference is $0$. 

We compare three different procedures: The first is the kernel density difference estimator \eqref{eq:den_diff_kernel} under joint tuning (``Joint"), the second the same estimator under separate tuning (``Separate"), and the third the least-squares density difference estimator of \cite{sugiyama2013density} (``LSDD"), which minimizes a regularized least-squares objective in a Gaussian reproducing kernel Hilbert space and is known to attain optimal rates. For kernel difference methods, we exploit higher-order kernels constructed via Legendre polynomials \citep{tsybakov2008introduction}, tuning the kernel order $\gamma \in \{2,4,6,8\}$. All methods share the same bandwidth grid $h \in \{0.2, 0.3, \ldots, 2\}$. We compare the pointwise bias and empirical standard error evaluated at $u=0.5$ for sample sizes $n \in \{100, 200, \ldots, 2000\}$, averaged over 1,000 replications of the DGP.

\begin{figure}[ht!]
        \centering        \includegraphics[width=0.7\textwidth]{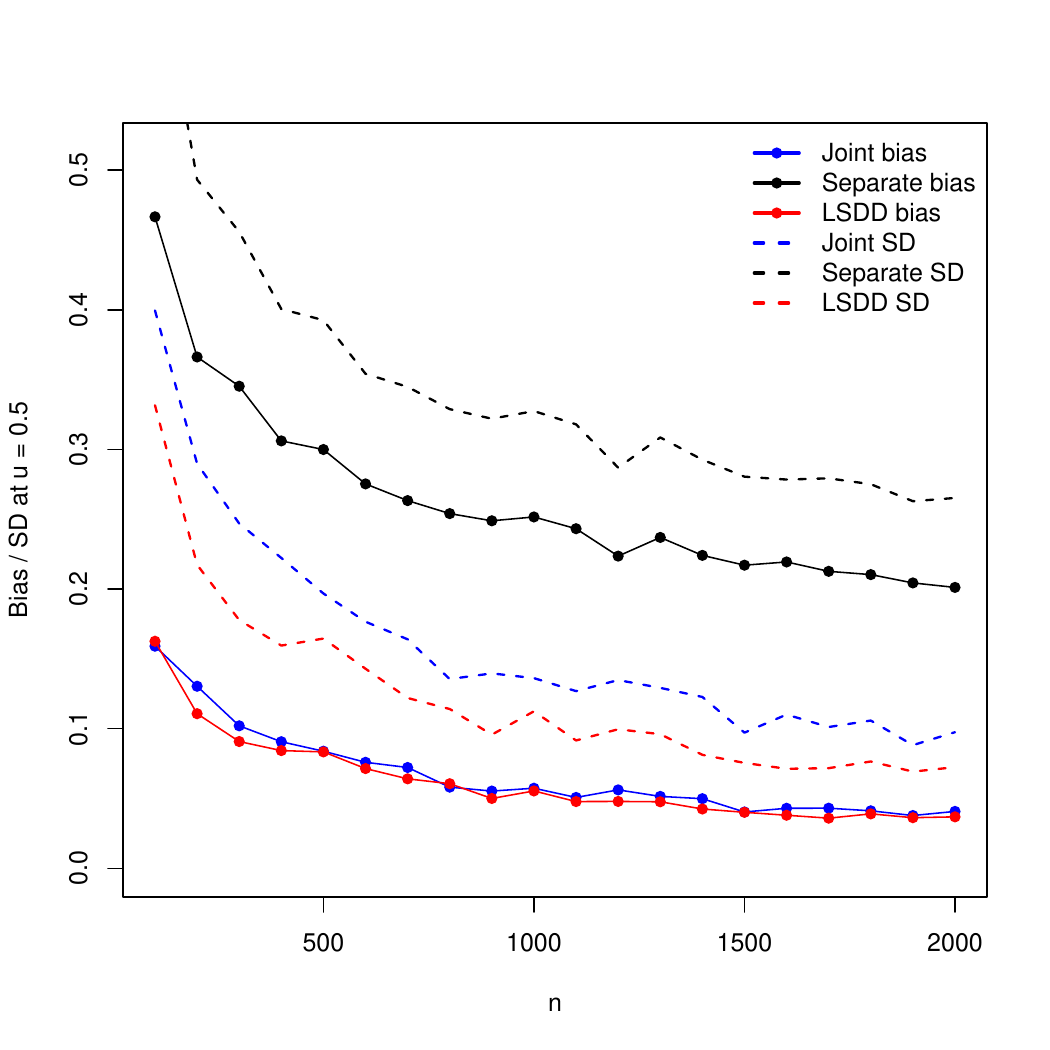}
    \caption{Comparison of different procedures for estimating the density difference $\tau(u) = 0$.}
    \label{fig:sim2}
\end{figure}

Figure \ref{fig:sim2} displays the results. Note that the kernel estimator with joint tuning exhibits performance similar to that of LSDD for all $n$. We expect similar findings for density difference estimators based on other nonparametric methods (e.g., orthogonal series).

\section{Application: Conviction Rates by Attorney Type} \label{sec:application}

Granting equal access to competent legal representation, as guaranteed by the Sixth Amendment, is central to achieving criminal justice. However, the American system for indigent representation often falls short, even though up to 80\% of criminal defendants are assigned a public defender or a court-appointed private attorney, also known as an \emph{appointed counsel}, because they cannot afford an attorney themselves \citep{pollitzworden2014criminal}. The case of Surrency \citep{bach2009ordinary}, an appointed counsel in Georgia known for his ``efficient" approach of negotiating plea deals, well illustrates the problem -- 
many of his clients, as a result of his tactic, plead guilty without realizing they could go to jail. In contrast, salaried public defenders are often better motivated and resourced, although they too must contend with excessively high caseloads \citep{bacak2024stress, bacak2020fighting, welch2019sources}. Yet, defendants rarely know which system will handle their case, let alone have any choice in that matter. 

Appointed counsel are typically assigned on an ad hoc, case-by-case basis, and usually paid under flat fee contracts \citep{spangenbert1995indigent, neubauer2015america}. Accordingly, they come from diverse legal backgrounds and have an incentive to dispose of cases quickly, unlike salaried public defenders. As a result, assigned counsel systems ``have been criticized for appointing attorneys with inadequate skills, experience, and qualifications to represent indigent defendants." \citep{cohen2014better}. Hence, a natural question that arises is: \emph{Are appointed counsel associated with worse case outcomes than public defenders?} 

Relevant empirical literature \citep{anderson2011difference, cohen2014better, roach2014indigent} suggest that appointed counsel are indeed associated with less favorable case outcomes, but their evidence is based on limited data. Many studies \citep{cohen2014better, roach2014indigent} base their analyses on the State Court Processing Statistics (SCPS) from the Bureau of Justice Statistics, which only covers felony cases in 75 most populous counties in the U.S. Others not relying on SCPS \citep{anderson2011difference} still primarily analyze data from large counties. 

We therefore leverage novel data from the Measures for Justice (MFJ), a nonprofit whose mission is to tackle this ``criminal justice data crisis". Our analysis centers on Pennsylvania, where court-related decisions around indigent defense are decentralized and administered by the counties, unlike states that have a statewide public defender office. Variation across counties is thus significant and meaningful, making counties the most appropriate units of analysis. Whereas SCPS only covers Allegheny, Philadelphia, and Montgomery, MFJ provides data for all 67 counties and includes cases beyond felonies. 

Specifically, MFJ supplies county-level counts of convictions and total cases by attorney type during 2009--2013. To address our question, we estimate the densities of conviction rates by attorney type, which is a binomial mixing density estimation problem where a conviction is treated as a ``success". The extreme variation in case counts by county and attorney type -- ranging from 1 to 24,740 for appointed counsel, 154 to 117,282 for public defenders, and 54 to 39,295 for private attorneys -- motivates studying the binomial mixing density estimation problem under trial heterogeneity.

Figure \ref{fig:conv} displays the centered kernel density estimates of the conviction rates by attorney type, with 95\% pointwise intervals obtained via undersmoothing. Bandwidths are selected by Lepski's method using the \texttt{goldenshluger\_lepski()} function from R package $\texttt{KDE}$, implemented based on \cite{goldenshluger2011bandwidth}. 
The vertical dashed lines mark the means. Estimated conviction rates for appointed counsel are generally highest ($\text{mean}=0.82$), next highest for public defenders ($\text{mean}=0.77$), and lowest for private attorneys ($\text{mean}=0.68$). At first glance, the pattern reported in earlier SCPS-based studies appears to extend beyond the largest counties and felony cases.

\begin{figure}[ht]
    \centering
    \includegraphics[width=0.7\linewidth]{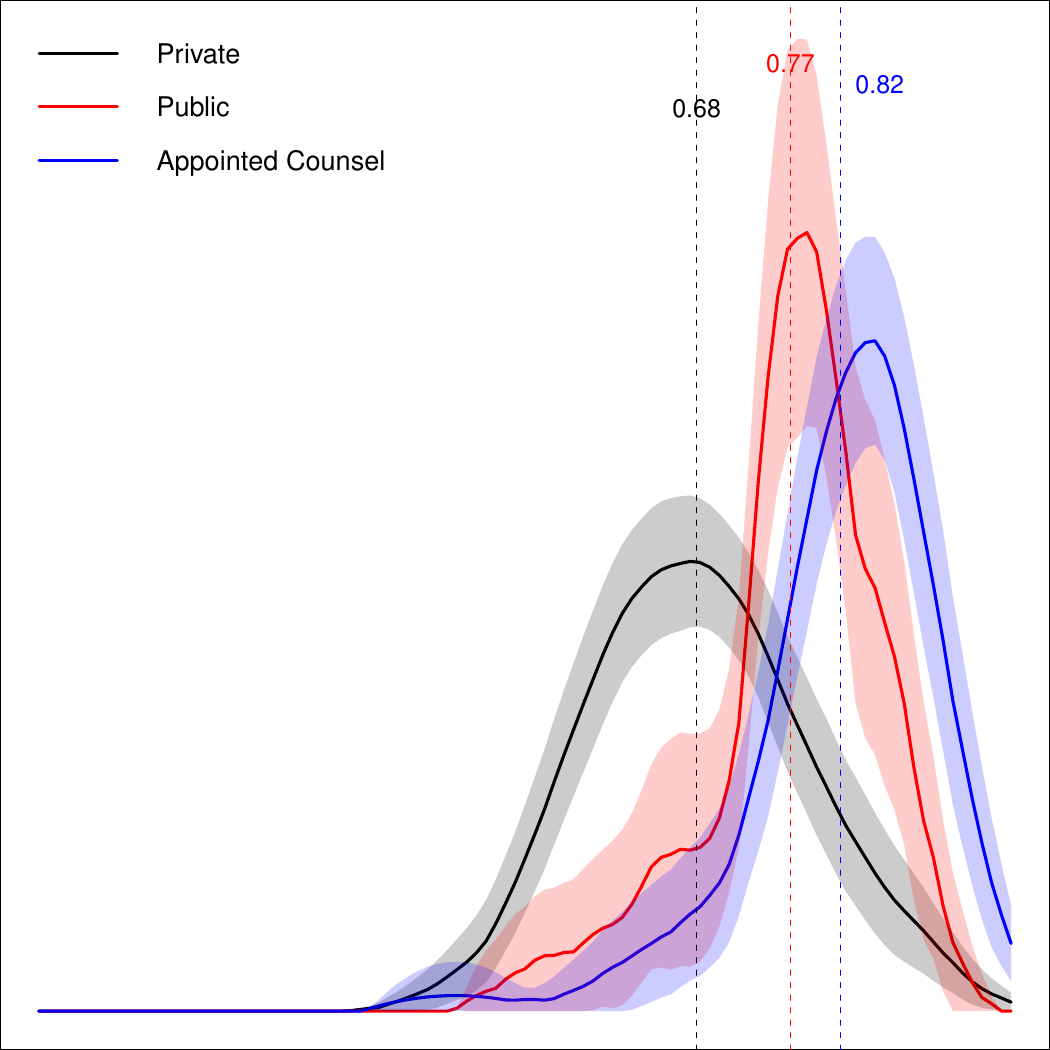}
    \caption{Kernel density estimates of conviction rates by attorney type.}
    \label{fig:conv}
\end{figure}

\paragraph{Potential Confounding.} Are appointed counsel actually less effective at getting clients acquitted than public defenders? We note a potential source of \emph{confounding}, which is that appointed counsel more often represent defendants facing more severe charges. Figure \ref{fig:nonv} shows the density estimates for nonviolent misdemeanor and felony cases with two different bail options: release on recognizance (ROR) and cash bail. ROR refers to cases in which the defendant is not required to post any cash bail, but must promise to attend all future court proceedings; it is therefore the least restrictive type of bail. While various factors determine the type of bail set, such as the defendant’s financial status and criminal record, the seriousness of the alleged offense is typically the most important deciding factor. Therefore, cases with ROR likely involve less severe cases than those requiring cash bail. 

\begin{figure}[ht!]
    \centering
    \begin{subfigure}[b]{0.48\textwidth}
        \centering
        \includegraphics[width=\textwidth]{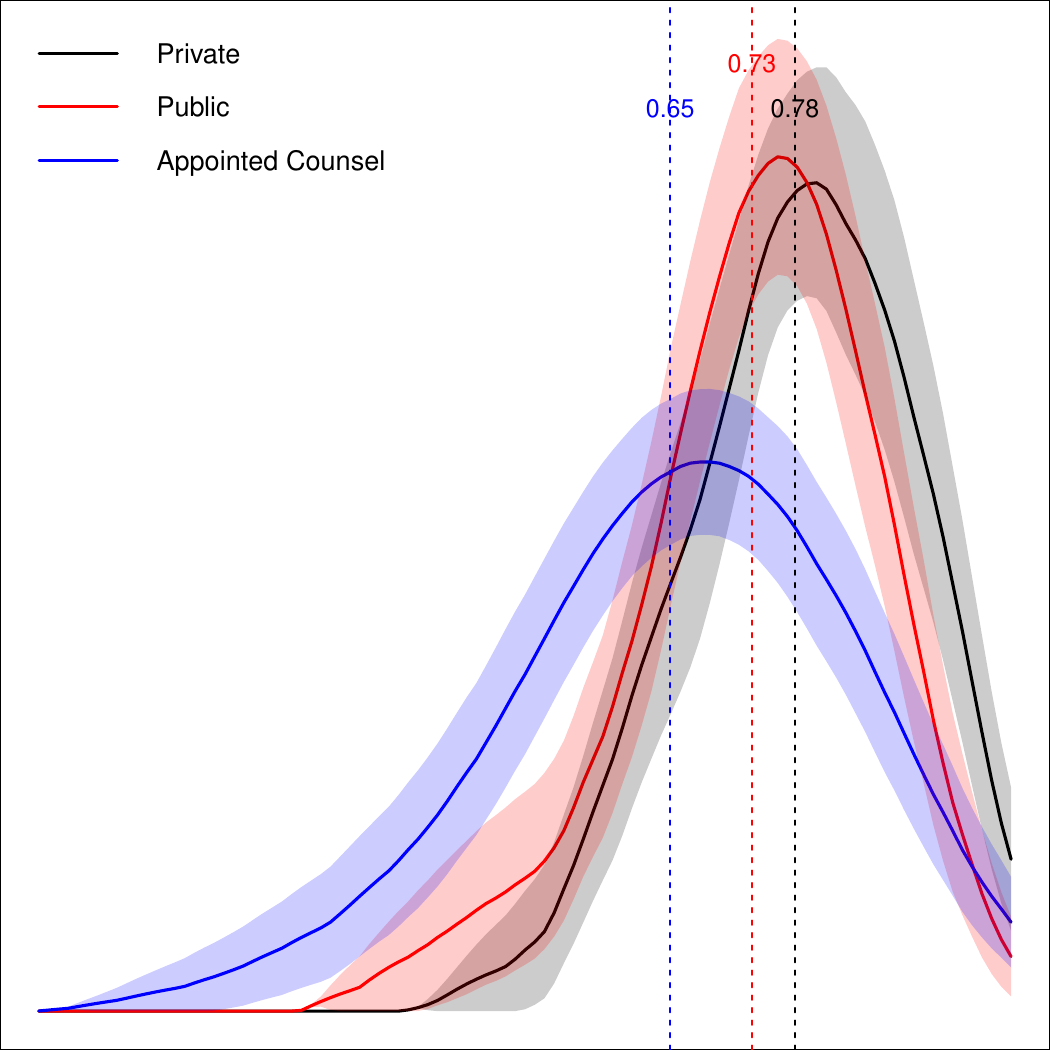}
        \caption{Nonviolent misdemeanor with ROR.}
        \label{fig:nonv_misd_ror}
    \end{subfigure}
    \hfill
    \begin{subfigure}[b]{0.48\textwidth}
        \centering
        \includegraphics[width=\textwidth]{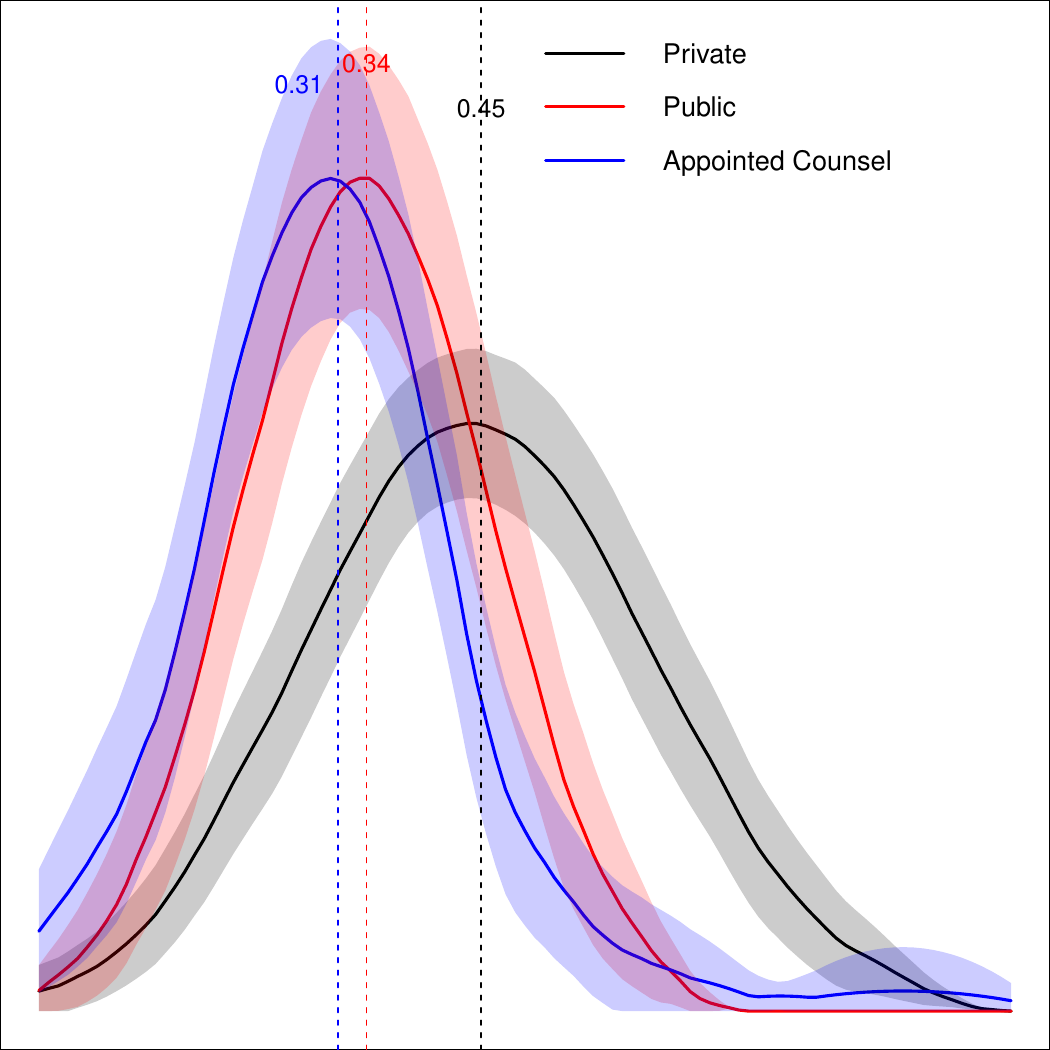}
        \caption{Nonviolent felony with ROR.}
        \label{fig:felony_ror}
    \end{subfigure}
    \vfill
    \begin{subfigure}[b]{0.48\textwidth}
        \centering
        \includegraphics[width=\textwidth]{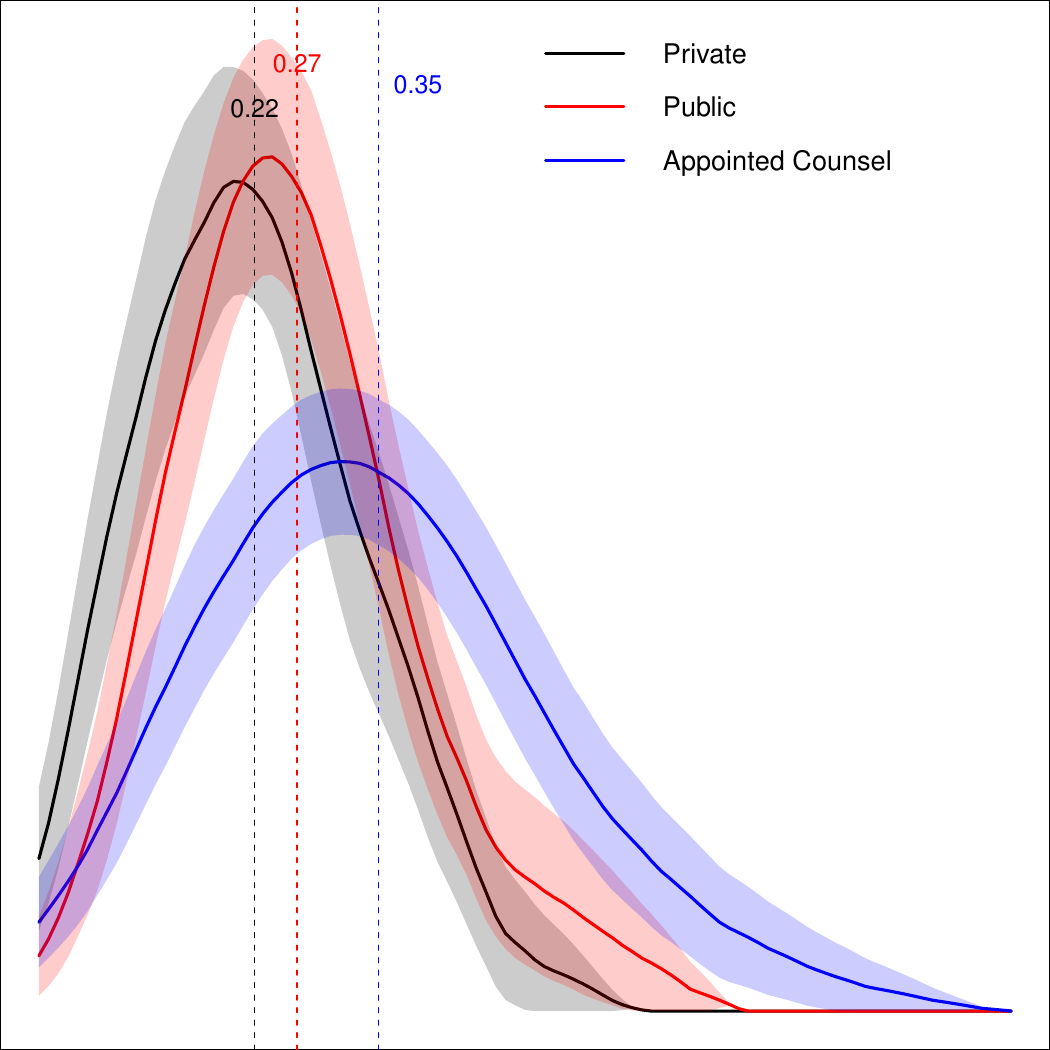}
        \caption{Nonviolent misdemeanor with cash bail.}
        \label{fig:misd_mon_bail}
    \end{subfigure}
    \hfill
    \begin{subfigure}[b]{0.48\textwidth}
        \centering
        \includegraphics[width=\textwidth]{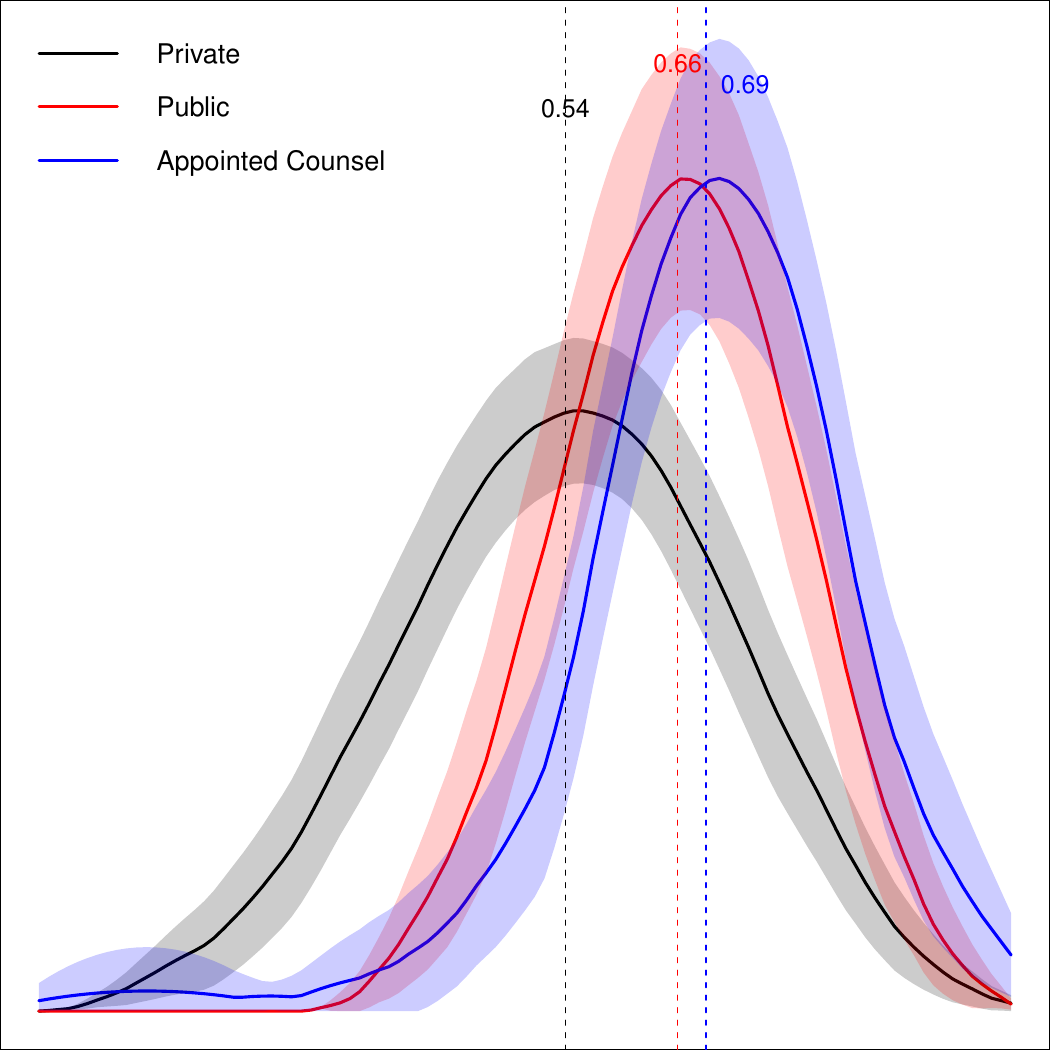}
        \caption{Nonviolent felony with cash bail.}
        \label{fig:felony_mon_bail}
    \end{subfigure}
    \caption{Kernel density estimates for the proportions of nonviolent misdemeanors and felonies by attorney type.}
    \label{fig:nonv}
\end{figure}

We observe that appointed counsel are more frequently involved in cases with cash bail for both misdemeanors and felonies, 
which could have contributed to the observed discrepancies in conviction rates. Therefore, counties should review and \emph{standardize} the largely ad-hoc process of assigning appointed counsel, ensuring more serious cases are not systematically routed to them. 

Future research in this direction should also examine several additional possibilities. First, appointed counsel may take on too many cases in response to the fee-based system, thus unable to devote adequate time and resources to cases. Second, they may also represent cases with higher cash bail if they see the client for the first time after the bail hearing has already been completed, or if ``exogenous fluctuations in the local labor market induce relatively more low-quality attorneys to choose assigned counsel work" \citep{roach2014indigent}. Finally, appointed counsel may prioritize non-indigent clients or lack administrative support to ensure they meet clients early in the case to represent them effectively \citep{anderson2011difference}.

\section{Conclusion}
\label{sec:conclusion}

In this work, we have studied nonparametric estimation of $s$-smooth binomial mixing densities under trial heterogeneity. We have shown that the pointwise bias of the KDE based on empirical proportions is $\bigO(h^s + 1/\sqrt{\tilde{t}} + 1/(h\tilde{t}))$, where $\tilde{t}$ denotes the harmonic mean of the trial parameters. When reduced to the homogeneous case, this is faster than the rate in \cite{ye2021binomial}. In addition, under higher-order smoothness, we have established that a centered KDE further improves the additional bias from empirical proportions by replacing the $1/\sqrt{\tilde t}$ term with $1/\tilde t$. We have also studied the difference between two $s$-smooth densities and have shown that one can exploit usual nonparametric density estimators and still achieve optimal performance. Density difference may also prove useful in contexts beyond those studied to date, e.g., in answering complex causal questions if comparing the difference in counterfactual densities \citep{kennedy2023semiparametric}.

Empirically, we have examined the conviction rates by attorney type in Pennsylvania using novel criminal justice data with broader coverage. In particular, we note that the estimates are generally higher for appointed counsel than for public defenders, potentially because appointed counsel are taking on more severe cases at higher rates. As aggregate data are common in criminal justice, the KDEs can be usefully applied in other contexts, such as comparing private and public prisons \citep{bacak2018availability}. 

There are many intriguing directions for future research. The first is to investigate actual minimax optimality for estimating smooth binomial mixing densities under trial heterogeneity. Our results characterize when KDE-based estimators recover the usual nonparametric density estimation rates, but sharp minimax rates for this problem remain unknown. 

The second question is uncertainty quantification. 
In our setting, a more informative interval than the one based on undersmoothing would explicitly reflect the number of trials. Intuitively, the interval should contract when the number of trials is large and expand when it is small. One option is a plug-in interval that uses the pointwise bias bound from Theorem \ref{thm:bound_on_error_kernel_het}. However, the constants regarding $p$ are difficult to estimate in practice. 

Finally, one can incorporate covariates. For example, adopting the setup of \cite{ignatiadis2019covariate}, we can posit the model
\begin{equation*}
   X_i \mid Q_i = q \sim \Bin(t_i, q), \quad \operatorname{logit}(Q_i) \mid C_i=c \sim N(m(c), \sigma^2),
\end{equation*}
where $C_i$ denotes the covariate vector for unit $i$. A natural question is whether we can estimate $Q_i$ more accurately under assumptions on $m(\cdot)$. 

The problem of binomial mixing density estimation clearly presents many captivating challenges.

\paragraph{Acknowledgments.}
The authors thank David S. Choi, Woonyoung Chang, Mateo Dulce Rubio, Lucas Kania, Zhenghao Zeng, Erik A. Bensen, and Soheun Yi for useful comments during the development of this work. The authors also thank Yandi Shen for suggesting relevant literature.

\bibliography{references}
\appendix

\section{Proof of Main Results}\label{sec:proof}
In this section, we present proofs of the results introduced in the main text. 
\subsection{Proof of Lemma \ref{lem:bernstein_error_bound}}
\begin{proof}
Let $B_{xt_i} \sim \textnormal{Beta}(x+1, t_i-x+1)$. We can write
\begin{align}
     \cE(f_i) &= \sum_{x=0}^{t_i}f(x/t_i) \binom{t_i}{x} \int q^{x}(1-q)^{t_i-x}p(q)dq -\int f(q)p(q)dq \nonumber \\
     &= \sum_{x=0}^{t_i} f(x/t_i)\frac{\E\{p(B_{xt_i})\}}{t_i+1} - \int f(q)p(q)dq \nonumber \\ 
     &= \sum_{x=0}^{t_i} f(x/t_i)\frac{\E\{p(B_{xt_i})\} - p\{\E(B_{xt_i})\}}{t_i+1} + \sum_{x=0}^{t_i} f(x/t_i)\left\{\frac{p\big(\frac{x+1}{t_i+2}\big)}{t_i+1}-\frac{p(x/t_i)}{t_i}\right\} \label{eq:1} \\
     &+ \left\{\sum_{x=0}^{t_i} f(x/t_i)\frac{p(x/t_i)}{t_i} - \int f(q)p(q)dq \right\}. \label{eq:2}
\end{align}
Note that
\begin{align*}
    \big|\E\{p(B_{xt_i}) - p\{\E(B_{xt_i})\}\big|&= \left|\E\{p(B_{xt_i}) - p\bigg(\frac{x+1}{t_i+2}\bigg)\right| \\
&\leq \E \left|p(B_{xt_i}) - p\bigg(\frac{x+1}{t_i+2}\bigg) \right|\\
&\leq L\E\left|B_{xt_i} - \frac{x+1}{t_i+2} \right|^\alpha \\
&\leq L\{\var(B_{xt_i})\}^{\alpha/2} = L\left\{\frac{(x+1)(t_i-x+1)}{(t_i+2)^2(t_i+3)}\right\}^{\alpha/2} \leq L\left(\frac{1/4}{t_i+3}\right)^{\alpha/2},
\end{align*}
where the third result follows from condition \ref{cond:p_holder_cont}, the fourth from H\"{o}lder's inequality and $B_{xt_i} \sim \textnormal{Beta}(x+1, t_i-x+1)$, and the last from the fact that $(x+1)(t_i-x+1) \leq (t_i+2)^2/4$ for any $x \in \{0,\ldots, t_i\}.$ Hence the absolute value of the first term in \eqref{eq:1} can be bounded as
\begin{equation*}
    L\left(\frac{1/4}{t_i+3}\right)^{\alpha/2}\sum_{x=0}^{t_i}\frac{|f(x/t_i)|}{t_i+1}.
\end{equation*}
For the second term in \eqref{eq:1}, we have
\begin{align*}
    &\left| \sum_{x=0}^{t_i} f(x/t_i)\left\{\frac{p\big(\frac{x+1}{t_i+2}\big)}{t_i+1} - \frac{p(x/t_i)}{t_i}\right\}\right| \\
    &= \left| \sum_{x=0}^{t_i} f(x/t_i)\left[\bigg\{\frac{p\big(\frac{x+1}{t_i+2}\big) - p(x/t_i)}{t_i+1}\bigg\} + p(x/t_i)\bigg(\frac{1}{t_i+1}-\frac{1}{t_i}\bigg)\right]\right| \\
&\leq L\sum_{x=0}^{t_i}\frac{|f(x/t_i)|}{t_i+1}\left|\frac{x+1}{t_i+2}-x/t_i\right|^\alpha + \frac{p_{\max}}{t_i}\sum_{x=0}^{t_i}\frac{|f(x/t_i)|}{t_i+1} \\
&\leq \left\{L\bigg(\frac{1}{t_i+2}\bigg)^\alpha + \frac{p_{\max}}{t_i}\right\}\sum_{x=0}^{t_i}\frac{|f(x/t_i)|}{t_i+1},
\end{align*}
where the second line follows from condition \ref{cond:p_holder_cont} and the third from $\left(x+1)/(t_i+2) -x/t_i \right| \leq 1/(t_i+2)$ is maximized when $x \in \{0, t_i\}$. The proof is complete by noting that the absolute value of \eqref{eq:2} corresponds to $\cR(f_ip)$.
\end{proof}

\subsection{Proof of Theorem \ref{thm:bound_on_error_kernel_het}}\label{sec:proof_thm1}
\begin{proof}
We first consider the bias. Note that
\begin{equation*}
    |\E\{\widehat{p}_h(u) - \tilde{p}_h(u) + \tilde{p}_h(u) - p(u)\}| \leq |\E\{\tilde{p}_h(u) - p(u)\}| + |\E\{\widehat{p}_h(u) - \tilde{p}_h(u)\}|,
\end{equation*}
where $\tilde{p}_h(u) \coloneqq 1/n\sum_{i=1}^n K\{(Q_i-u)/h\}/h$ is the KDE constructed from true proportions. Hence the first term is the usual pointwise bias that can be bounded by $(LB/\floor{s}!)h^s$ under conditions \ref{cond:K_int_1}--\ref{cond:K_higher_bdd} and the H\"{o}lder assumption on $p$. The second term is the bias induced by empirical proportions that, by Lemma \ref{lem:bernstein_error_bound}, bounded by the quantity
\begin{align*}
    \frac{1}{n}\sum_{i=1}^n \left[\left\{L\left(\frac{1/4}{t_i+3}\right)^{\alpha/2} + L\left(\frac{1}{t_i+2}\right)^\alpha + \frac{p_{\max}}{t_i}\right\}\sum_{x=0}^{t_i}\frac{|K_h(x/t_i)|}{t_i+1} + \cR(K_{h,i}p)\right].
\end{align*}
We now introduce a proposition that bounds the quasi-Riemann sum approximation error \eqref{eq:quasi-riemann} when $f=K_h$.
\begin{proposition}\label{prop:bound_on_riemann}
    Assume conditions of Lemma \ref{lem:bernstein_error_bound} and additionally
    \begin{enumerate}
        \item $|K(v)| \leq K_{\max}\ind(|v| \leq 1)$, and 
        \item $|K(v)-K(v')| \leq M|v-v'|^{\beta}$ for some $0 < \beta \leq 1$ and for all $v,v' \in [-1,1]$.
    \end{enumerate}
    Then
    \begin{align*}
         \cR(K_{h,i}p) 
        &\leq \frac{K_{\max}p_{\max}}{ht_i} + \left(2+\frac{1}{ht_i}\right)\left\{LK_{\max}\left(\frac{1}{t_i}\right)^\alpha + 2Mp_{\max}\left(\frac{1}{ht_i}\right)^\beta \right\}.
    \end{align*}
\end{proposition}

\begin{proof}
 Note that, since $K$ and $p$ are continuous, mean value theorem implies
    \begin{align*}
        \int \frac{1}{h}K\bigg(\frac{q-u}{h}\bigg)p(q)dq &= \sum_{x=1}^{t_i} \int_{(x-1)/t_i}^{x/t_i} \frac{1}{h}K\bigg(\frac{q-u}{h}\bigg)p(q)dq \\ 
        &= \frac{1}{t_i}\sum_{x=1}^{t_i}\frac{1}{h}K\bigg(\frac{x/t_i-\xi_{x_i}-u}{h}\bigg)p(x/t-\xi_{x_i})
    \end{align*}
    for some $\xi_{x_i} \in [0,1/t_i]$ so that 
    \begin{align}
        &\sum_{x=0}^{t_i}\frac{1}{ht_i}K\bigg(\frac{x/t_i-u}{h}\bigg)p(x/t_i)-\int\frac{1}{h}K\bigg(\frac{q-u}{h}\bigg)p(q)dq \nonumber \\
        &= \frac{1}{ht_i}K\bigg(\frac{0-u}{h}\bigg)p(0) + \frac{1}{ht_i}\sum_{x=1}^{t_i} K\bigg(\frac{x/t_i-u}{h}\bigg)\big\{p(x/t_i)-p(x/t_i-\xi_{x_i})\big\} \label{eq:3} \\ 
        &+ \frac{1}{ht_i}\sum_{x=1}^{t_i}\left\{K\bigg(\frac{x/t_i-u}{h}\bigg)-K\bigg(\frac{x/t_i-\xi_{x_i}-u}{h}\bigg)\right\}p(x/t_i-\xi_{x_i}) \label{eq:4}.
    \end{align}
    We can bound the absolute value of the first term in \eqref{eq:3} as
    \begin{equation*}
        \left|\frac{1}{ht_i}K\bigg(\frac{-u}{h}\bigg)p(0)\right| \leq \frac{K_{\max}p_{\max}}{ht_i},
    \end{equation*}
    and that of the second term in \eqref{eq:3} as
    \begin{align*}
        \bigg|&\frac{1}{ht_i}\sum_{x=1}^{t_i} K\bigg(\frac{x/t_i-u}{h}\bigg)\{p(x/t_i)-p(x/t_i-\xi_{x_i})\}\bigg| \\
        & \leq \frac{K_{\max}}{ht_i}\sum_{x=1}^{t_i}\ind\bigg(\bigg|\frac{x/t_i-u}{h}\bigg| \leq 1 \bigg)|p(x/t_i)-p(x/t_i-\xi_{x_i})| \\
        &\leq \frac{LK_{\max}}{ht_i}\bigg(\frac{1}{t_i}\bigg)^{\alpha}\sum_{x=1}^{t_i}\ind\bigg(\bigg|\frac{x/t_i-u}{h}\bigg| \leq 1 \bigg) \leq LK_{\max}\bigg(\frac{1}{t_i}\bigg)^{\alpha}\bigg(2+\frac{1}{ht_i}\bigg),
    \end{align*}
    where the first inequality follows from condition \ref{cond:K_bdd}, the second from condition \ref{cond:p_holder_cont} and the fact that $0 \leq \xi_{x_i} \leq 1/t_i$, and the third since there exists at most $2t_ih + 1$ separated points in the interval $t_iu \pm t_ih$. 

    For \eqref{eq:4}, note that
    \begin{align*}
        &\left|\frac{1}{ht_i}\sum_{x=1}^{t_i}\left\{K\bigg(\frac{x/t_i-u}{h}\bigg)-K\bigg(\frac{x/t_i-\xi_{x_i}-u}{h}\bigg)\right\}p(x/t_i-\xi_{x_i})\right| \\
        &\leq \frac{p_{\max}}{ht_i}\sum_{x=1}^{t_i}\left|K\bigg(\frac{x/t_i-u}{h}\bigg)-K\bigg(\frac{x/t_i-\xi_{x_i}-u}{h}\bigg)\right|\ind\bigg\{\min\bigg(\bigg|\frac{x/t_i-\xi_{x_i}-u}{h}\bigg|, \bigg| \frac{x/t_i-u}{h} \bigg|\bigg) \leq 1\bigg\} \\
        &\leq \frac{Mp_{\max}}{ht_i}\bigg(\frac{1}{ht_i}\bigg)^\beta \sum_{x=1}^{t_i}\left\{\ind \bigg(\bigg|\frac{x/t_i-\xi_{x_i}-u}{h}\bigg|\leq 1\bigg)+\ind \bigg(\bigg|\frac{x/t_i-u}{h}\bigg|\leq 1\bigg)\right\} \\
        &\leq 2Mp_{\max}\bigg(\frac{1}{ht_i}\bigg)^\beta\bigg(2+\frac{1}{ht_i}\bigg),
    \end{align*}
    where the second inequality follows from from condition \ref{cond:K_holder_cont} and $\ind\{\min(a,b) \leq 1\} \leq \ind\{a \leq 1\} + \ind\{b \leq 1\}$, and the third from the equally spaced points argument. 
    It follows that
    \begin{equation*}
        \cR(f_ip) \leq \frac{p_{\max}K_{\max}}{ht_i}+\bigg(2+\frac{1}{ht_i}\bigg)\bigg\{LK_{\max}\bigg(\frac{1}{t_i}\bigg)^{\alpha} + 2Mp_{\max}\bigg(\frac{1}{ht_i}\bigg)^{\beta}\bigg\}.\qedhere
    \end{equation*}
\end{proof}
Putting everything together, we obtain
\begin{align*}
     |\E&\{\widehat{p}_h(u)\}-p(u)| \leq \frac{LB}{\floor{s}!}h^s \\
     &+ \frac{1}{n}\sum_{i=1}^n\Bigg[\left\{L\left(\frac{1/4}{t_i+3}\right)^{\alpha/2} + L\left(\frac{1}{t_i+2}\right)^\alpha + \frac{p_{\max}}{t_i}\right\}\sum_{x=0}^{t_i}\frac{|K_h(x/t_i)|}{t_i+1} \\
     &+\frac{K_{\max}p_{\max}}{ht_i} + \left(2+\frac{1}{ht_i}\right)\left\{LK_{\max}\left(\frac{1}{t_i}\right)^\alpha + 2Mp_{\max}\left(\frac{1}{ht_i}\right)^\beta \right\}\Bigg],
\end{align*}
which is the desired bias bound. 

For the pointwise variance, we have via similar arguments that
\begin{align*}
    \var\{\widehat{p}_h(u)\} &\leq \frac{1}{n^2h^2}\sum_{i=1}^n\E\bigg\{K\bigg(\frac{X_i/t_i-u}{h}\bigg)^2\bigg\} \\
    &= \frac{1}{n^2h^2}\sum_{i=1}^n\sum_{x=0}^{t_i} K\bigg(\frac{x/t_i-u}{h}\bigg)^2\binom{t_i}{x} \int q^x(1-q)^{t_i-x}p(q)dq \\
    &\leq \frac{K_{\max}^2}{n^2h^2}\sum_{i=1}^n\sum_{x=0}^{t_i}\ind\big(t_iu-t_ih \leq x \leq t_iu+t_ih\big) \frac{p_{\max}}{t_i+1} \\
    &\leq \frac{K_{\max}^2p_{\max}}{n^2h^2}\sum_{i=1}^n\frac{2t_ih+1}{t_i+1} \leq \frac{K^2_{\max}p_{\max}}{nh}\bigg(2+\frac{1}{h\tilde{t}}\bigg). \qedhere
\end{align*}
\end{proof}

\subsection{Proof of Corollary \ref{cor:bound_on_error_kernel_het}}\label{sec:proof_cor1}
\begin{proof}
First, note that
\begin{align*}
    \sum_{x=0}^{t_i} \frac{|K_h(x/t_i)|}{t_i+1} \leq \frac{K_{\max}}{h(t_i+1)}\sum_{x=0}^{t_i} \ind(t_iu-t_ih \leq x \leq t_iu+t_ih) \leq \bigg(2+\frac{1}{ht_i}\bigg)K_{\max},
\end{align*}
i.e., the quantity $\sum_{x=0}^{t_i}|K_h(x/t_i)|/(t_i+1)$ is bounded under condition \ref{cond:K_bdd} of Theorem \ref{thm:bound_on_error_kernel_het}.
Therefore, if $\sum_{x=0}^{t_i}|K_h(x/t_i)|/(t_i+1) \leq C''$ for some $C''>0$ and $\alpha=1$ (part of condition \ref{cond:p_K_lipschitz}), then
\begin{align*}
\frac{1}{n}\sum_{i=1}^n&\left[\left\{L\bigg(\frac{1/4}{t_i+3}\bigg)^{1/2} + L\bigg(\frac{1}{t_i+2}\bigg) + \frac{p_{\max}}{t_i}\right\}\sum_{x=0}^{t_i}\frac{|K_h(x/t_i)|}{t_i+1} + \cR(K_{h,i}p)\right] \\
&\leq \frac{1}{n}\sum_{i=1}^n \left\{C''(L/2+L+p_{\max})\bigg(\frac{1}{\sqrt{t_i}}\bigg) + \cR(K_{h,i}p) \right\} \\ 
&\leq C''(2L+p_{\max})\bigg(\frac{1}{\sqrt{\tilde{t}}}\bigg) + \frac{1}{n}\sum_{i=1}^n\cR(K_{h,i}p).    
\end{align*}
where the last step follows from Jensen's inequality.
Now we give a concise bound on the quasi-Riemann sum approximation error. Under conditions \ref{cond:p_K_lipschitz} and \ref{cond:bounded_h}, we have
    \begin{align*}
       \frac{1}{n}\sum_{i=1}^n\cR(K_{h,i}p) &= \frac{1}{n}\sum_{i=1}^n\left[\frac{p_{\max}K_{\max}}{ht_i} + \left(2 + \frac{1}{ht_i}\right)\left\{LK_{\max}\left(\frac{1}{t_i}\right) + 2Mp_{\max}\left(\frac{1}{ht_i}\right)\right\}\right] \\ 
        &\leq \frac{p_{\max}K_{\max}}{h\tilde{t}} + (2 + 1/C)\left\{LK_{\max}\left(\frac{1}{\tilde{t}}\right) + 2Mp_{\max}\left(\frac{1}{h\tilde{t}}\right)\right\} \\
        &\leq \bigg\{p_{\max}K_{\max}+(2+1/C)(LK_{\max}C' + Mp_{\max})\bigg\}\left(\frac{1}{h\tilde{t}}\right).
    \end{align*} 
Putting both together, it follows that
\begin{align*}
    |\E\{\widehat{p}_h(u)\} - p(u)| &\leq \frac{LB}{\floor{s}!}h^s + C''(2L+p_{\max})\bigg(\frac{1}{\sqrt{\tilde{t}}}\bigg) \\
    &+ \bigg\{p_{\max}K_{\max}+(2+1/C)(LK_{\max}C' + Mp_{\max})\bigg\}\left(\frac{1}{h\tilde{t}}\right).
\end{align*}
For the variance, condition \ref{cond:bounded_h} again implies
\begin{equation*}
    \var\{\widehat{p}_h(u)\} \leq \frac{K^2_{\max}p_{\max}}{nh}\bigg(2+\frac{1}{h\tilde{t}}\bigg) \leq K^2_{\max}p_{\max}(2+1/C)\bigg(\frac{1}{nh}\bigg). \qedhere
\end{equation*}
\end{proof}

\subsection{Proof of Lemma \ref{lem:centered_bernstein_error_bound}}

\begin{proof}
Note that $\cE_\mu(f_i)$ can be decomposed as
\begin{align*}
\cE_\mu(f_i)
&= \sum_{x=0}^{t_i}f_i(\mu_{i,x})\frac{\E{p(B_{xt_i})}-p(\mu_{i,x})}{t_i+1} + \sum_{x=0}^{t_i}\frac{f_i(\mu_{i,x})p(\mu_{i,x})}{t_i+1} - \int_0^1 f_i(q)p(q)dq,
\end{align*}
where $B_{xt_i}\sim \text{Beta}(x+1,t_i-x+1)$. The absolute value of the last two terms corresponds to the centered quasi-Riemann sum approximation error $\cR_\mu(f_ip)$.

By the first-order Taylor expansion, since $p'$ is Lipschitz with constant $L'$, we have
\[
    \left|p(B_{xt_i}) - p(\mu_{i,x}) - p'(\mu_{i,x})(B_{xt_i}-\mu_{i,x})\right|
    \leq \frac{L'}{2}(B_{xt_i}-\mu_{i,x})^2.
\]
Taking expectations yields
\[
\left|\E\{p(B_{xt_i})\}-p(\mu_{i,x})\right| \leq \frac{L'}{2}\E{(B_{xt_i}-\mu_{i,x})^2} = \frac{L'}{2}\var(B_{xt_i}) \leq \frac{L'}{8(t_i+3)},
\]
where the last result follows from
\[
    \var(B_{xt_i}) = \frac{(x+1)(t_i-x+1)}{(t_i+2)^2(t_i+3)} \leq \frac{1}{4(t_i+3)}.
\]
It follows that
\[
    |\cE_\mu(f_i)| 
    \leq \sum_{x=0}^{t_i}|f_i(\mu_{i,x})|\frac{|\E{p(B_{xt_i})}-p(\mu_{i,x})|}{t_i+1} + \cR_\mu(f_ip) 
    \leq \frac{L'}{8(t_i+3)}\sum_{x=0}^{t_i}\frac{|f_i(\mu_{i,x})|}{t_i+1} + \cR_\mu(f_ip),
\]
and subsequently
\[
\left|\frac{1}{n}\sum_{i=1}^n \cE_\mu(f_i)\right|
\leq \frac{1}{n}\sum_{i=1}^n \left\{\frac{L'}{8(t_i+3)}\sum_{x=0}^{t_i}\frac{|f_i(\mu_{i,x})|}{t_i+1} + \cR_\mu(f_ip)\right\}. \qedhere 
\]
\end{proof}

\subsection{Proof of Corollary \ref{cor:bound_on_error_kernel_het_centered}}
\begin{proof}
As for the KDE \eqref{eq:kde}, we begin by introducing a corollary that bounds $\cR_\mu(f_ip)$ with $f=K_h$ and shows that it is of same order as its uncentered counterpart.

\begin{corollary}\label{cor:centered_riemann_sum_error}
Assume the conditions of Lemma \ref{lem:centered_bernstein_error_bound} and Proposition \ref{prop:bound_on_riemann}. Then 
    \begin{align*}
         \cR_\mu(K_{h,i}p) 
        &\leq \frac{2K_{\max}p_{\max}}{ht_i} + \left(2+\frac{1}{ht_i}\right)\left\{LK_{\max}\left(\frac{1}{t_i}\right)^\alpha + 2Mp_{\max}\left(\frac{1}{ht_i}\right)^\beta \right\}.
    \end{align*}
\end{corollary}

\begin{proof}
As in the proof of Proposition~\ref{prop:bound_on_riemann}, we can write by continuity of $K$ and $p$ that
\begin{align*}
\int_0^1 \frac{1}{h}K\bigg(\frac{q-u}{h}\bigg)p(q)dq
&= \sum_{x=0}^{t_i+1}
\int_{x/(t_i+2)}^{(x+1)/(t_i+2)}
\frac{1}{h}K\bigg(\frac{q-u}{h}\bigg)p(q)dq \\
&= \frac{1}{t_i+2}\sum_{x=0}^{t_i+1}
\frac{1}{h}K\bigg(\frac{\mu_{i,x}-\xi_{i,x}-u}{h}\bigg)p(\mu_{i,x}-\xi_{i,x}),
\end{align*}
for some $\xi_{i,x}\in[0,1/(t_i+2)]$. 
It follows that
\begin{align*}
\frac{1}{t_i+1}\sum_{x=0}^{t_i}
&\frac{1}{h}K\bigg(\frac{\mu_{i,x}-u}{h}\bigg)p(\mu_{i,x})
- \int_0^1 \frac{1}{h}K\bigg(\frac{q-u}{h}\bigg)p(q)dq \\
&= \bigg(\frac{1}{t_i+1} - \frac{1}{t_i+2}\bigg)
\sum_{x=0}^{t_i}
\frac{1}{h}K\bigg(\frac{\mu_{i,x}-u}{h}\bigg)p(\mu_{i,x}) \\
&\quad + \frac{1}{t_i+2}
\left\{\sum_{x=0}^{t_i}
\frac{1}{h}K\bigg(\frac{\mu_{i,x}-u}{h}\bigg)p(\mu_{i,x}) - \sum_{x=0}^{t_i}\frac{1}{h}K\bigg(\frac{\mu_{i,x}-\xi_{i,x}-u}{h}\bigg)p(\mu_{i,x}-\xi_{i,x})\right\} \\
&\quad - \frac{1}{t_i+2}\frac{1}{h}K\bigg(\frac{1-\xi_{i,t_i+1}-u}{h}\bigg)p(1-\xi_{i,t_i+1}).
\end{align*}
The absolute value of the third term is bounded by
$K_{\max}p_{\max}/\{h(t_i+2)\} \leq K_{\max}p_{\max}/(ht_i)$.
For the first term, we have
\begin{align*}
    &\left|\bigg(\frac{1}{t_i+1} - \frac{1}{t_i+2}\bigg)\sum_{x=0}^{t_i}\frac{1}{h}K\bigg(\frac{\mu_{i,x}-u}{h}\bigg)p(\mu_{i,x})\right| \\
    &\quad \leq\frac{1}{(t_i+1)(t_i+2)}(t_i+1)\frac{K_{\max}p_{\max}}{h} = \frac{K_{\max}p_{\max}}{h(t_i+2)}
    \leq \frac{K_{\max}p_{\max}}{ht_i}.
\end{align*}

The absolute value of the second term can be handled as in the proof of Proposition~\ref{prop:bound_on_riemann}, with $\mu_{i,x}$ in place of $x/t_i$ and $\xi_{i,x}\in[0,1/(t_i+2)]$ in place of $[0,1/t_i]$, yielding
\[
    \left(2 + \frac{1}{ht_i}\right)
    \left\{LK_{\max}\left(\frac{1}{t_i}\right)^\alpha + 2Mp_{\max}\left(\frac{1}{ht_i}\right)^\beta\right\}.
\]
Combining the three bounds gives
\[
    \cR_\mu(K_{h,i}p) \leq \frac{2K_{\max}p_{\max}}{ht_i} + \left(2 + \frac{1}{ht_i}\right)\left\{LK_{\max}\left(\frac{1}{t_i}\right)^\alpha + 2Mp_{\max}\left(\frac{1}{ht_i}\right)^\beta\right\}. \qedhere
\]
\end{proof}

Now the proof mostly follows from the proofs of Theorem \ref{thm:bound_on_error_kernel_het} and Corollary \ref{cor:bound_on_error_kernel_het}. For the bias, recall the decomposition
\[
    |\E\{\widehat p_{\mathrm{c},h}(u)\}-p(u)| \leq |\E\{\tilde{p}_h(u) - p(u)\}| + |\E\{\widehat{p}_{\mathrm{c},h}(u) - \tilde{p}_h(u)\}|,
\]
where the first term is $\lesssim h^s$. For the second term, Lemma \ref{lem:centered_bernstein_error_bound} gives
\begin{equation*}
\left|\frac1n \sum_{i=1}^n \cE_{\mu}(f_i)\right|
\leq \frac1n \sum_{i=1}^n \left\{\frac{L'}{8(t_i+3)}\sum_{x=0}^{t_i}\frac{|f_i(\mu_{i,x})|}{t_i+1} + \cR_\mu(f_ip)\right\},
\end{equation*}
By the proof of Corollary \ref{cor:bound_on_error_kernel_het}, we have
$\sum_{x=0}^{t_i}|K_{h,i}(\mu_{i,x})|/(t_i+1) \lesssim 1+1/(ht_i)$ and 
\[
\frac{1}{n}\sum_{i=1}^n\frac{1}{t_i+3}\sum_{x=0}^{t_i}\frac{|K_{h,i}(\mu_{i,x})|}{t_i+1}
\lesssim \frac{1}{n}\sum_{i=1}^n\left(\frac{1}{t_i}+\frac{1}{ht_i^2}\right) 
\lesssim \frac{1}{\tilde t}+\frac{1}{h\tilde t}.
\]
Next, we know from Corollary \ref{cor:centered_riemann_sum_error} that $\cR_\mu(K_{h,i}p)$ is of the same order as $\cR(K_{h,i}p)$ and thus by the same logic in Corollary \ref{cor:bound_on_error_kernel_het},
\[
\frac{1}{n}\sum_{i=1}^n \cR_\mu(K_{h,i}p)
\lesssim
\frac{1}{\tilde t}+\frac{1}{h\tilde t}.
\]
It follows that 
\[
|\E\{\widehat p_{\mathrm{c},h}(u)\}-p(u)| \lesssim h^s+\frac{1}{\tilde t}+\frac{1}{h\tilde t}.
\]

Similarly for the variance, we have
\[
\var\{\widehat p_{\mathrm{c},h}(u)\} \lesssim \frac{1}{nh}\left(1+\frac{1}{h\tilde t}\right) \lesssim
\frac{1}{nh},
\]
following the arguments in the proofs of Theorem \ref{thm:bound_on_error_kernel_het} and Corollary \ref{cor:bound_on_error_kernel_het}.
\end{proof}

\subsection{Proof of Proposition \ref{prop:lepski}}
We first prove a lemma that shows $p_{h,h^*} \lesssim (1/h^* - 1/h)/n$.
\begin{lemma}\label{lem:phh}
Assume the conditions of Theorem \ref{thm:bound_on_error_kernel_het} and additionally
\begin{enumerate}
    \item $K$ is Lipschitz with $K(-1)=K(1)=0$,
    \item $C/\tilde{t} \leq h \leq C'$ for some $C,C' > 0$.
\end{enumerate}
Then
\[
p_{h,h^*} \lesssim \frac{1}{n}\left(\frac{1}{h^*}-\frac{1}{h}\right).
\]
\end{lemma}

\begin{proof}
First note that, as in the proof of Theorem \ref{thm:bound_on_error_kernel_het}, we can bound $p_{h,h^*}$ as 
\begin{align}
p_{h,h^*} &\leq \frac{1}{n^2}\sum_{i=1}^n\E\left[\left\{K_{h^*}(X_i/t_i)-K_{h}(X_i/t_i)\right\}^2\right] \nonumber \\ 
&\leq \frac{p_{\max}}{n^2}\sum_{i=1}^n \sum_{x=0}^{t_i}\frac{\left\{K_{h^*}(x/t_i)-K_{h}(x/t_i)\right\}^2}{t_i+1}. \label{eq:lem3_summand}
\end{align}
We bound \eqref{eq:lem3_summand} in two ways. First, the crude variance bound used in the proof of Theorem \ref{thm:bound_on_error_kernel_het} gives
\[
\sum_{x=0}^{t_i}\frac{\left\{K_{h^*}(x/t_i)-K_{h}(x/t_i)\right\}^2}{t_i+1} \lesssim \frac{1}{h^*}\left(1+\frac{1}{h^*t_i}\right),
\]
and it follows by $h^*\gtrsim 1/\tilde{t}$ that
\[
    \frac{1}{n}\sum_{i=1}^n\sum_{x=0}^{t_i}\frac{\left\{K_{h^*}(x/t_i)-K_{h}(x/t_i)\right\}^2}{t_i+1} \lesssim \frac{1}{h^*}.
\]

Second, let 
\[
\delta \coloneqq \frac{1}{h^*} - \frac{1}{h} \eqqcolon a - b \geq 0,\]
with $a \geq b$ since $h\geq h^*$. For $z\in[0,1]$, we claim that
\[
    \left|aK\{a(z-u)\}-bK\{b(z-u)\}\right| \lesssim \delta \ind(|z-u|\leq h).
\]
If $|z-u|>h$, both terms are zero. If $|z-u|\leq h^*$, then both lie in $[-1,1]$, and Lipschitz continuity gives
\begin{align*}
    \left|aK\{a(z-u)\}-bK\{b(z-u)\}\right|
    &\leq |a-b|\,|K\{a(z-u)\}| + b\left|K\{a(z-u)\}-K\{b(z-u)\}\right| \\
    &\leq \delta + b\delta |z-u| \lesssim \delta. 
\end{align*}
If $h^*<|z-u|\leq h$, then the first term is zero and 
\[
    b|K\{b(z-u)\}| \leq Mb\left(1-b|z-u|\right)
    = M\frac{h-|z-u|}{h^2}
    \leq M\frac{h-h^*}{h^2} 
    = M\delta\frac{h^*}{h} \lesssim \delta,
\]
where the first inequality uses $K(-1)=K(1)=0$ and Lipschitz continuity. This proves the claim.

Now applying the claim with $z=x/t_i$ and the equally spaced point argument, we obtain
\[
    \sum_{x=0}^{t_i}\frac{\left\{K_{h^*}(x/t_i)-K_{h}(x/t_i)\right\}^2}{t_i+1}
    \lesssim\delta^2\sum_{x=0}^{t_i}\frac{\ind(|x/t_i-u|\leq h)}{t_i+1}
    \lesssim \delta^2\left(h+\frac{1}{t_i}\right),
\]
and subsequently
\[
\frac{1}{n}\sum_{i=1}^n\sum_{x=0}^{t_i}\frac{\left\{K_{h^*}(x/t_i)-K_{h}(x/t_i)\right\}^2}{t_i+1} \lesssim \delta^2 h+\frac{\delta^2}{\tilde{t}}.
\]
Hence combining the two bounds for \eqref{eq:lem3_summand} yields
\[
    \frac{1}{n}\sum_{i=1}^n\sum_{x=0}^{t_i}\frac{\left\{K_{h^*}(x/t_i)-K_{h}(x/t_i)\right\}^2}{t_i+1}
    \lesssim \min\left\{\frac{1}{h^*},\delta^2 h+\frac{\delta^2}{\tilde{t}}\right\}.
\]

Now we show that the minimum actually simplifies to $\delta$. First, note that $\delta^2/\tilde{t} \lesssim \delta^2 h^* \leq \delta$ since $h^*\gtrsim 1/\tilde{t}$ and $\delta\leq 1/h^*$. It remains to show that
$\min\left\{1/h^*, \delta^2 h\right\} \lesssim \delta.$ If $\delta h\leq 1$, then $\delta^2 h\leq \delta$
so the minimum is at most \(\delta\). If $\delta h>1$, then $1/h<\delta$ and $1/h^{*} = \delta+ 1/h < 2\delta$ so the minimum is at most $2\delta$. 

It follows that 
\[
    \frac{1}{n}\sum_{i=1}^n \sum_{x=0}^{t_i}\frac{\left\{K_{h^*}(x/t_i)-K_{h}(x/t_i)\right\}^2}{t_i+1}
\lesssim \delta,
\]
and subsequently
\[
    p_{h,h^*} \lesssim \frac{1}{n}\delta = \frac{1}{n}\left(\frac{1}{h^*}-\frac{1}{h}\right). \qedhere
\]
\end{proof}

\begin{proof}
We first verify \eqref{eq:lepski_bound}. For $h,h'\in\mathcal H$, write
\[
    \xi_{h,h'} = \widehat p_{h'}(u)-\E\{\widehat p_{h'}(u)\} -\widehat p_h(u)+\E\{\widehat p_h(u)\},
\]
so that
$
\widehat p_{h'}(u)-\widehat p_h(u)=b_{h,h'}+\xi_{h,h'}.
$
Consider the events
\[
    E_1 = \left\{\max_{h'\in\mathcal H^-(h^*)} \frac{|\xi_{h^*,h'}|}{\sqrt{p_{h^*,h'}}} \leq c_{h^*}(\alpha)\right\}, \quad \text{and} \quad
    E_2 = \left\{\max_{h\in\mathcal H^+(h^*)}\frac{|\xi_{h,h^*}|}{\sqrt{p_{h,h^*}}}\le c^-_{h^*}(\alpha)\right\}.
\]
By definition of the critical values, $P(E_1\cap E_2)\ge 1-2\alpha$. On $E_1$, since $h^*$ satisfies $H_{h^*}$,
\[
    \max_{h'\in\mathcal H^-(h^*)}\frac{|\widehat p_{h'}(u)-\widehat p_{h^*}(u)|}{\sqrt{p_{h^*,h'}}} 
    \leq 1+c_{h^*}(\alpha),
\]
so $H_{h^*}$ is not rejected and therefore $\widehat h\ge h^*$ by the selection rule. On $E_2$, for any
\[
    h\in\mathcal H^0 = \left\{h\in\mathcal H^+(h^*): |b_{h,h^*}|^2 > p_{h,h^*}\{1+c_h(\alpha)+c^-_{h^*}(\alpha)\}^2\right\},
\]
we have
\[
    \frac{|\widehat p_{h^*}(u)-\widehat p_h(u)|}{\sqrt{p_{h,h^*}}} \geq \frac{|b_{h,h^*}|}{\sqrt{p_{h,h^*}}} - \frac{|\xi_{h,h^*}|}{\sqrt{p_{h,h^*}}} > 1+c_h(\alpha).
\]
Thus $H_h$ is rejected for all $h\in\mathcal H^0$. Hence, we have $\widehat h\in\mathcal H_{\mathrm{IN}} = \{h\in\mathcal H:h\ge h^*\}\setminus\mathcal H^0$ on $ E_1\cap E_2$. Finally, since $H_{\widehat h}$ is not rejected and $\widehat h\ge h^*$, we have
$
|\widehat p_{\widehat h}(u)-\widehat p_{h^*}(u)| \leq \sqrt{p_{\widehat h,h^*}}\{1+c_{\widehat h}(\alpha)\},
$
which implies \eqref{eq:lepski_bound} since $\widehat h\in\mathcal H_{\mathrm{IN}}$.

Now Lemma \ref{lem:phh} gives
\begin{align*}
\max_{h\in\mathcal H_{\mathrm{IN}}}
\sqrt{p_{h,h^*}}\{1+c_h(\alpha)\}
&\lesssim \sqrt{\frac{\Delta_{\mathrm{IN}}}{n}}\max_{h\in\mathcal H_{\mathrm{IN}}}\{1+c_h(\alpha)\}.
\end{align*}
Now, note that the critical values satisfy $\max_{h\in\mathcal H_{\mathrm{IN}}}c_h(\alpha) \lesssim  \sqrt{\log |\cH|}$ under a standard maximal inequality for the finite grid, e.g., Lemma 2.3 of \citet{massart2007concentration}. Since $\log|\mathcal H| = \mathcal O(\log n)$ by assumption, with probability at least $1-2\alpha$, we have
\[
|\widehat p_{\widehat h}(u)-\widehat p_{h^*}(u)|
\lesssim
\sqrt{\frac{\Delta_{\mathrm{IN}}\log n}{n}}.
\]
Plugging in $\Delta_{\mathrm{IN}}\asymp 1/h^*$ or  $\Delta_{\mathrm{IN}} = \mathcal O(1)$ yields the desired.
\end{proof}

\subsection{Proof of Proposition \ref{prop:clt}}
\begin{proof}
We have the decomposition
\begin{equation*}
    \widehat{p}_h(u) - p(u) = \big[\widehat{p}_h(u) - \E\{\widehat{p}_h(u)\}\big] + \big[\E\{\widehat{p}_h(u)\}- p(u)\big].
\end{equation*}
To show that the first term is asymptotically normal, we introduce a lemma that provides a sufficient condition for Lindeberg's condition.

\begin{lemma}[A sufficient condition for Lindeberg's condition]

    Suppose $\{X_{ni}, n \geq 1, 1 \leq i \leq i_n\}$ is a triangular array such that, for each $n$, $X_{n1},\ldots,X_{ni_n}$ are independent. Let $B_n^2 = \sum_{i=1}^{i_n} \var(X_{ni})$. Assume there exists $\{L_n, n \geq 1\}$ satisfying 
\begin{equation*}
    \max_{1 \leq i \leq i_n} |X_{ni}| \leq L_n, \quad \frac{L_n}{B_n} \to 0.
\end{equation*}
Then, Lindeberg's condition holds, i.e., for any $\epsilon > 0$, we have
\begin{equation*}
    \lim_{n \to \infty}\frac{1}{B_n^2}\sum_{i=1}^{i_n} \E\big[\{X_{ni} - \E(X_{ni})\}^2\ind\{|X_{ni}-\E(X_{ni})| \geq \epsilon B_n\}\big] = 0.
\end{equation*}
As a result,
\begin{equation*}
    \frac{\sum_{i=1}^{i_n}\{X_{ni}-\E(X_{ni})\}}{B_n} \overset{d}{\to} N(0,1).
\end{equation*}
\end{lemma}


We verify the sufficient condition for Lindeberg's condition with $i_n=n$ and
\[
X_{ni}^{\star}=\frac{1}{h}K\left(\frac{Z_i^\star-u}{h}\right).
\]
First, note that
\begin{equation*}
    |X_{ni}^\star| \leq \frac{1}{h}K_{\max}\ind\left(\bigg|\frac{X_i/t_i-u}{h}\bigg| \leq 1 \right) \leq \frac{K_{\max}}{h}.
\end{equation*}
We also know that
\[
B_{n,\star}^2 \coloneqq \sum_{i=1}^n\var(X_{ni}^{\star}) = n^2\var\{\widehat p_{\star,h}(u)\} \asymp \frac{n}{h}.
\]
It follows that
\[
\frac{\max_{1\le i\le n}|X_{ni}^{\star}-\E(X_{ni}^{\star})|}{B_{n,\star}}
\lesssim
\frac{1/h}{\sqrt{n/h}}
=
\frac{1}{\sqrt{nh}}
\to 0,
\]
and hence Lindeberg's condition holds, implying the asymptotic normality of the first term.

Next, since $nh\to\infty$, the condition $\sqrt{nh}h^s\to0$ implies $h \to 0$ and $\sqrt{nh}(h\tilde t)^{-1} \to 0$ implies $h\tilde t\to\infty$. Therefore, the bandwidth window required for the simplified bias bounds in Corollaries \ref{cor:bound_on_error_kernel_het} and \ref{cor:bound_on_error_kernel_het_centered} holds eventually, hence the bias is $o(nh)^{-1/2})$ under the corresponding undersmoothing conditions. 

It remains to show consistency of the variance estimator. 
Since $\sum_{i=1}^n\var(X_{ni}^{\star}) \asymp nh^{-1}$ and $|X_{ni}^{\star}|\le K_{\max}/h$, we have by the triangular-array law of large numbers that
\[
\frac{\sum_{i=1}^n\left[\{X_{ni}^{\star}-\E(X_{ni}^{\star})\}^2 - \var(X_{ni}^{\star})\right]}{\sum_{i=1}^n\var(X_{ni}^{\star})} = o_{\mathbb P}(1).
\]
By analogous arguments as in the proofs of Corollaries
\ref{cor:bound_on_error_kernel_het} and
\ref{cor:bound_on_error_kernel_het_centered}, we have
\[
\sum_{i=1}^n {\E(X_{ni}^{\star})}^2
\lesssim
n\left(1+\frac{1}{h\tilde t}+\frac{1}{h^2\tilde t}\right) = o(nh^{-1}),
\]
where the last equality follows from $h\to0$ and $h\tilde t\to\infty$. It follows that
\[
\sum_{i=1}^n
\{X_{ni}^{\star}-\widehat p_{\star,h}(u)\}^2 = \sum_{i=1}^n\var(X_{ni}^{\star})\{1+o_{\mathbb P}(1)\}.
\]
Consequently,
\[
\frac{\widehat{\var}\{\widehat p_{\star,h}(u)\}}{\var\{\widehat p_{\star,h}(u)\}}
=
\frac{\{n(n-1)\}^{-1}\sum_{i=1}^n \{X_{ni}^{\star}-\widehat p_{\star,h}(u)\}^2}{n^{-2}\sum_{i=1}^n\var(X_{ni}^{\star})} \overset{p}{\to} 1.
\]
Applying Slutsky's theorem gives asymptotic normality with the estimated variance.
\end{proof}

\subsection{Proof of Proposition \ref{prop:den_diff}}\label{sec:proof_prop_den_diff}
\begin{proof}
Let $A^n \coloneqq (A_1,\ldots,A_n)$. Note that
\begin{align*}
    \E\{\tilde\tau_h(u) \mid A^n\} &= \frac{1}{n}\sum_{i=1}^n \frac{A_i}{\overline{A}}\int K_h(q)p(q \mid A_i)dq - \frac{1}{n}\sum_{i=1}^n \frac{1-A_i}{1-\overline{A}}\int K_h(q)p(q \mid A_i)dq \\
    &= \frac{1}{n}\sum_{i:A_i=1} \frac{A_i}{\overline{A}}\int K_h(q)p_1(q)dq - \frac{1}{n}\sum_{i:A_i=0} \frac{1-A_i}{1-\overline{A}}\int K_h(q)p_0(q)dq \\
    &=\int K_h(q)\{p_1(q)-p_0(q)\}dq = \int K_h(q)\tau(q)dq.
\end{align*}
Therefore, the conditional bias $\E\{\tilde\tau_h(u) \mid A^n\} - \tau(u)$ is bounded by $(L_\tau B/\floor{\gamma}!)h^\gamma$ following standard arguments as in Theorem \ref{thm:bound_on_error_kernel_het}.

For the conditional variance, we have
\begin{align*}
    \var\{\tilde{\tau}_h(u) \mid A^n\} &\leq \frac{1}{nh^2}\bigg(\frac{A_i}{\overline{A}} - \frac{1-A_i}{1-\overline{A}}\bigg)^2 \E\bigg\{K\bigg(\frac{Q_i-u}{h}\bigg)^2 \biggm| A^n \bigg\} \\
    &\leq \max\bigg\{\frac{1}{(1-\overline{A})^2}, \frac{1}{\overline{A}^2}\bigg\}\frac{p_{\max}B'}{nh} \leq \frac{p_{\max}B'}{\epsilon^2}\bigg(\frac{1}{nh}\bigg). \qedhere
\end{align*}
\end{proof}

\subsection{Proof of Proposition \ref{prop:joint_cv}}
\begin{proof} 
We condition on $D_1^n$ and $\{A_j:j\in D_2^n\}$ so that $\tilde\tau_h$ and $\overline A$ are fixed. For each $h\in\mathcal H$,
\begin{align*} 
\E\{\widehat R(h)\mid D_1^n,\{A_j:j\in D_2^n\}\} &= \int \tilde\tau_h(u)^2du - 2\left\{ \int \tilde\tau_h(u)p_1(u)du - \int \tilde\tau_h(u)p_0(u)du \right\}
\\ &= R(h)-\int \tau(u)^2du,
\end{align*} 
so $\widehat R(h)$ is an unbiased pseudo-risk for $R(h)$ up to an additive constant independent of $h$. 

Now for notational simplicity, for $j\in D_2^n$, let
\[
W_j=\frac{A_j}{\overline A}-\frac{1-A_j}{1-\overline A}, 
\]
and
\[
Z_{j,h} =  \left\{2W_j\tau(Q_j)-\int \tau(u)^2du\right\} - \left\{2W_j\tilde\tau_h(Q_j)-\int \tilde\tau_h(u)^2du\right\},
\]
Further, define
\[
\mathbb P_2 Z_h \coloneqq
\frac{1}{|D_2^n|}\sum_{j\in D_2^n} Z_{j,h},
\qquad
\mathbb P Z_h
\coloneqq \frac{1}{|D_2^n|}\sum_{j\in D_2^n} \E\{Z_{j,h}\mid D_1^n,{A_j:j\in D_2^n}\}.
\]
Note that 
\[
\mathbb P Z_h=R(h), \qquad
\mathbb P_2 Z_h
= \widehat R(h) + \left\{\frac{2}{|D_2^n|}\sum_{j\in D_2^n}W_j\tau(Q_j) -\int \tau(u)^2du\right\},
\]
so minimizing $\mathbb P_2Z_h$ over $h\in\mathcal H$ is equivalent to minimizing $\widehat R(h)$, which implies that $\widehat h$ also satisfies
\[
\widehat h=\arg\min_{h\in\mathcal H}\mathbb P_2Z_h.
\]

Now decompose 
\[
R(\widehat h) = \mathbb P Z_{\widehat h} = (1+\delta)\mathbb P_2Z_{\widehat h} + \left\{(1+\delta)(\mathbb P-\mathbb P_2)Z_{\widehat h} - \delta \mathbb P Z_{\widehat h}\right\} \eqqcolon S_n+T_n.
\]
Since $\widehat h$ minimizes $\mathbb P_2Z_h$, we have
\[
S_n = (1+\delta)\mathbb P_2Z_{\widehat h} \le (1+\delta)\mathbb P_2Z_{h^*},
\]
and subsequently
\[
\E\{S_n\mid D_1^n,{A_j:j\in D_2^n}\} \le (1+\delta)\mathbb P Z_{h^*} = (1+\delta)R(h^*).
\]
For $T_n$, we need to check the uniform boundedness and Bernstein variance conditions of $Z_{j,h}$ to hold. We have by assumption that 
\[
\|\tau\|_\infty\le \|p_1\|_\infty + \|p_0\|_\infty\le 2p_{\max},
\qquad
\sup_{h\in\mathcal H}\|\tilde\tau_h\|_\infty\le 2p_{\max},
\]
and $|W_j|\le \epsilon^{-1}$ since $\epsilon\le \overline A\le 1-\epsilon$. It follows that
\[
|Z_{j,h} - \mathbb PZ_h |\le \frac{16p_{\max}}{\epsilon} + 16p_{\max}^2.
\]
Further,
\begin{align*}
\var\{Z_{j,h}\mid D_1^n,{A_j:j\in D_2^n}\}
&\le 4W_j^2 \E\left[\{\tilde\tau_h(Q_j)-\tau(Q_j)\}^2\mid A_j \right] \\
&\le \frac{4p_{\max}}{\epsilon^2}\int \{\tilde\tau_h(u)-\tau(u)\}^2du \\
&= \frac{4p_{\max}}{\epsilon^2}R(h).
\end{align*} 
Then the Bernstein-union-bound argument in Theorem 7.1 of \citet{gyorfi2002distributionfree} yields
\[
\E\{T_n\mid D_1^n,{A_j:j\in D_2^n}\} \le C_\delta(p_{\max},\epsilon)\frac{1+\log|\mathcal H|}{n_2},
\]
where one may take
\[
C_\delta(p_{\max},\epsilon)
=
\frac{2p_{\max}(1+\delta)^2}{\delta\epsilon^2}
+
\frac{32}{3}(1+\delta)
\left(
\frac{p_{\max}}{\epsilon}+p_{\max}^2
\right).
\]
Putting the two bounds together and taking expectations gives the result.
\end{proof}



\subsection{Proof of Corollary \ref{cor:den_diff_het}}\label{sec:proof_cor_den_diff_het}
\begin{proof}
For notational ease, let $I_a \coloneqq \{i:A_i=a\}$ and $n_a \coloneqq |I_a|$. Define the group-specific Bernstein error
\[
    \cE_a(f_i) \coloneqq \int \bigg\{\sum_{x=0}^{t_i}f(x/t_i){t_i \choose x}q^x(1-q)^{t_i-x} - f(q)\bigg\}p_a(q)dq.
\]
Then
\begin{align*}
    \E\{\widehat{\tau}_h(u)\mid A^n\}-\tau(u)
    &=
    \bigg[
        \int K_h(q)\{p_1(q)-p_0(q)\}dq-\tau(u)
    \bigg]
    +
    \frac{1}{n_1}\sum_{i\in I_1}\cE_1(K_{h,i})
    -
    \frac{1}{n_0}\sum_{i\in I_0}\cE_0(K_{h,i})  \\
    &=
    \bigg\{
        \int K_h(q)\tau(q)dq-\tau(u)
    \bigg\}
    +
    \frac{1}{n_1}\sum_{i\in I_1}\cE_1(K_{h,i})
    -
    \frac{1}{n_0}\sum_{i\in I_0}\cE_0(K_{h,i}).
\end{align*}
The first term can be bounded by$(L_\tau B/\floor{\gamma}!)h^\gamma$. For the second and the third term, applying Lemma \ref{lem:bernstein_error_bound} separately to the observations in $I_a$ with $p$ replaced by $p_a$ yields
\[
    \left|
        \frac{1}{n_a}\sum_{i\in I_a}\cE_a(K_{h,i})
    \right|
    \leq
    \frac{1}{n_a}\sum_{i\in I_a}
    \bigg\{
        g_a(t_i,\alpha)
        \sum_{x=0}^{t_i}\frac{|K_h(x/t_i)|}{t_i+1}
        +
        r_a(h,t_i,\alpha,\beta)
    \bigg\},
    \quad a\in\{0,1\}.
\]
It follows that
\[
    |\E\{\widehat{\tau}_h(u)\mid A^n\}-\tau(u)|
    \leq
    \frac{L_\tau B}{\floor{\gamma}!}h^\gamma
    +
    \sum_{a=0,1}
    \frac{1}{n_a}\sum_{i\in I_a}
    \bigg\{
        g_a(t_i,\alpha)
        \sum_{x=0}^{t_i}\frac{|K_h(x/t_i)|}{t_i+1}
        +
        r_a(h,t_i,\alpha,\beta)
    \bigg\}.
\]

For the conditional variance, we have by independence that
\[
    \var\{\widehat{\tau}_h(u)\mid A^n\}
    =
    \var\{\widehat p_{1,h}(u)\mid A^n\}
    +
    \var\{\widehat p_{0,h}(u)\mid A^n\}.
\]
Applying the variance bound in Theorem \ref{thm:bound_on_error_kernel_het} to each group gives
\[
    \var\{\widehat p_{a,h}(u)\mid A^n\}
    \leq
    \frac{K_{\max}^2p_{\max}}{n_a h}
    \left(
        2+\frac{1}{h\tilde t_a}
    \right),
    \quad a\in\{0,1\}.
\]
It follows that
\begin{align*}
    \var\{\widehat{\tau}_h(u)\mid A^n\}
    &\leq
    \frac{K_{\max}^2p_{\max}}{h}
    \left[
        \frac{1}{n_1}
        \left(
            2+\frac{1}{h\tilde t_1}
        \right)
        +
        \frac{1}{n_0}
        \left(
            2+\frac{1}{h\tilde t_0}
        \right)
    \right] \\
    &=
    \frac{K_{\max}^2p_{\max}}{nh}
    \left[
        \frac{1}{\overline A}
        \left(
            2+\frac{1}{h\tilde t_1}
        \right)
        +
        \frac{1}{1-\overline A}
        \left(
            2+\frac{1}{h\tilde t_0}
        \right)
    \right] \\
    &\leq
    \frac{K_{\max}^2p_{\max}}{nh\epsilon}
    \left(
        4+\frac{1}{h\tilde t_1}
        +\frac{1}{h\tilde t_0}
    \right). \qedhere
\end{align*}
\end{proof}

\subsection{Proof of Corollary \ref{cor:den_diff_het_centered}}
\begin{proof}
The result follows from the proof of Corollary \ref{cor:den_diff_het}, with the groupwise empirical-proportion bounds replaced by the centered bounds of Corollary \ref{cor:bound_on_error_kernel_het_centered}, applied separately to $p_0$ and $p_1$.
\end{proof}

\end{document}